\documentclass[10pt]{article}
\usepackage{authblk}
\usepackage{graphicx} 
\usepackage{amsmath}
\usepackage{amsfonts}
\usepackage{amssymb}
\usepackage{amsthm}
\usepackage{float}
\usepackage{geometry}	

\allowdisplaybreaks

\usepackage[pdftex, backref=page]{hyperref} 
\renewcommand*{\backref}[1]{}
\renewcommand*{\backrefalt}[4]{%
    \ifcase #1 (Not cited.)%
    \or        (Cited on page~#2.)%
    \else      (Cited on pages~#2.)%
    \fi}

\usepackage[initials]{amsrefs}
\usepackage[english]{babel}

\usepackage{hyperref}
\usepackage{amssymb}
\usepackage{graphicx}
\usepackage{mathrsfs}
\usepackage{upref,amsthm,amsxtra,exscale}
\usepackage{cite}

\usepackage{caption}
\usepackage[shortlabels]{enumitem}
\usepackage{esint}
\setlist[description]{leftmargin=\parindent,labelindent=\parindent,itemsep=1pt,parsep=0pt,topsep=0pt}

\newtheorem{theorem}{Theorem}[section]

\newtheorem{remark}[theorem]{Remark}

\newtheorem{lemma}[theorem]{Lemma}

\numberwithin{equation}{section}

\usepackage{xcolor}

\def\d{\mathrm{\,d}}

\def\eps{\varepsilon}

\def\tilde{\widetilde}

\def\r{\mathbb{R}}

\let\div\undefined
\DeclareMathOperator{\div}{div}

\def\({\left(}
\def\){\right)}
\def\r{\mathbb{R}}
\def\r^n{\mathbb{R}^N}

\def\d{\mathrm{\,d}}

\def\tilde{\widetilde}

\def\d{\delta}

\def\p{\partial}

\def\a{\alpha}

\def\b{\beta}

\renewcommand{\leq}{\leqslant}
\renewcommand{\le}{\leqslant}
\renewcommand{\geq}{\geqslant}
\renewcommand{\ge}{\geqslant}

\author{Serena Dipierro}
\author{Jo{\~a}o Gon\c{c}alves da Silva}
\author{Giorgio Poggesi}
\author{Enrico Valdinoci}

\affil{ {\footnotesize Department of Mathematics and Statistics,} 
{\footnotesize The University of Western Australia,} \\
{\footnotesize 35 Stirling Highway,
Perth, WA 6009, Australia}\\
{\footnotesize\tt serena.dipierro@uwa.edu.au},
{\footnotesize\tt joao.goncalvesdasilva@research.uwa.edu.au},\\
{\footnotesize\tt giorgio.poggesi@uwa.edu.au},
{\footnotesize\tt enrico.valdinoci@uwa.edu.au}}

\title{A quantitative Gidas-Ni-Nirenberg-type result for the\\p-Laplacian via integral identities}
\date{}

\begin{document}
\maketitle

\begin{abstract}
We prove a quantitative version of a Gidas-Ni-Nirenberg-type symmetry result involving the $p$-Laplacian. 

Quantitative stability is achieved here via integral identities based on the proof of rigidity established by J. Serra in 2013, which extended to general dimension and the $p$-Laplacian operator an argument proposed by P. L. Lions in dimension~$2$ for the classical Laplacian. 

Stability results for the classical Gidas-Ni-Nirenberg symmetry theorem (involving the classical Laplacian) via the method of moving planes were established by Rosset in 1994 and by Ciraolo, Cozzi, Perugini, Pollastro in 2024. 

To the authors' knowledge, the present paper provides the first quantitative Gidas-Ni-Nirenberg-type result
involving the $p$-Laplacian for~$p \neq 2$.
Even for the classical Laplacian (i.e., for~$p=2$), this is the first time that integral identities are used to achieve stability for a Gidas-Ni-Nirenberg-type result.

In passing, we obtain a quantitative estimate for the measure of the singular set and an explicit uniform gradient bound.
\end{abstract}

\section{Introduction}
In~\cite{MR0544879}, B. Gidas, W. M. Ni and L. Nirenberg showed that positive classical 
solutions of the problem 
\begin{equation}\label{Original Problem}
\begin{cases}
-\Delta u = f(u) &  \text{ on }B,\\
 u = 0 & \text{ on }\partial B,
\end{cases}
\end{equation}
where~$B\subset \mathbb{R}^N$ (with~$N\geq 2$) is a ball, are radially symmetric and decreasing with respect to~$r = |x|$, provided that~$f$ can be written as the sum of a Lipschitz function and a nonincreasing function.  

The proof proposed in~\cite{MR0544879}
relied on the method of moving planes first introduced by Aleksandrov in~\cite{MR0102114}.
After the publication of~\cite{MR0544879}, several authors went on to generalize the result of~\cite{MR0544879} in several directions, see for instance,~\cite{MR1628044, MR1648566, KesavanPacella, serra2013radial, MR0653200, DPV_CVPDE, MR1159383, MR1190345, MR2096703}. 

In particular, P. L. Lions in~\cite{MR0653200} proved symmetry and monotonicity of positive solutions of~\eqref{Original Problem} in dimension~$2$ with a nonlinearity~$f$
which was assumed to be only locally bounded and strictly positive.
P. L. Lions' approach 
used a combination of the isoperimetric inequality and the Pohozaev identity and
was extended by Kesavan and Pacella in~\cite{KesavanPacella} to the $N$-Laplacian in~$\mathbb{R}^N$ for~$N\geq 2$.

Later on, J. Serra in~\cite{serra2013radial} extended such an approach to the $p-$Laplacian with~$p\in(1,+\infty)$ in all dimensions~$N\geq 2$. More, specifically, J. Serra showed radial symmetry and monotonicity of~$C^1(B)\cap C\left(\overline{B}\right)$ positive solutions of the problem
\begin{equation}\label{Serra's Problem}
   \begin{cases}
        -\Delta_p u:= -\div(|\nabla u|^{p-2}\nabla u) = f(u)& \text{ on }  B,\\
        u = 0 &\text{ on }\partial B,
    \end{cases}
\end{equation}
where~$B$ is a ball and~$f \in L^{\infty}_{\text{loc}}([0,+\infty))$ is nonnegative; for~$p\in(1,N)$ such a result also requires the assumption that there exists a nonincreasing function~$\phi:[0,+\infty)\rightarrow [0,+\infty)$ such that~$\phi\leq f\leq \frac{pN}{N-p}\phi$.  

Quantitative versions of the classical Gidas-Ni-Nirenberg result (involving the classical Laplacian) were obtained by Rosset in~\cite{MR1300801} and Ciraolo, Cozzi, Perugini, Pollastro in~\cite{MR4779387}; both these papers are based on the method of moving planes.\medskip

In this paper, we establish a quantitative version of J. Serra's result.
To the authors' knowledge, the present paper provides the first quantitative Gidas-Ni-Nirenberg-type result
involving the $p$-Laplacian for~$p \neq 2$.
We stress that, even for the classical Laplacian (i.e., for~$p=2$), this is the first time that a method relying on integral identities is used to achieve stability for a Gidas-Ni-Nirenberg-type result. In this regard, we mention that our quantitative analysis will show some similarities with the method used by Cianchi, Esposito, Fusco, Trombetti in~\cite{CianchiEspositoFuscoTrombetti+2008+153+189} to study a quantitative version of the Pólya-Szeg\"o principle.

We recall that the method used in~\cite{MR0653200, KesavanPacella, serra2013radial} was further extended to the anisotropic (and possibly weighted) setting in~\cite{DPV_CVPDE}: in this sense, we believe that our approach has the potential to be extended to such a setting (but, to present
the innovative aspects of our arguments in the clearest possible way, we focus here
on the isotropic setting, which is interesting in itself, and
we postpone the analysis of the anisotropic setting to forthcoming research).

We will work here mainly in the setting used in~\cite{serra2013radial}, but for the sake of simplicity, in our main result we will impose some suitable additional assumptions on the nonlinearity~$f$, which are, essentially, Lipschitz continuity and a positive lower bound~$\phi_0$ for~$f$. Nevertheless, we stress that our approach may be adapted to work with different choices of settings and parameters\footnote{More on this will be soon addressed in research to be presented in the near future. In particular, 
while writing this paper we learned that, independently of us, Ciraolo and Li have proved an alternative result which will appear in a forthcoming paper~\cite{CiraoloLi}.}, but for the sake of simplicity we preferred to focus on the setting described below. 
\medskip

The mathematical details of the framework considered are as follows.
We deal with~$p\in(1,+\infty)$ and the
$p$-Laplace equation
\begin{equation}\label{Problem}
   \begin{cases}
        -\Delta_p u = f(u) &  \text{ on } \Omega,\\
	  u>0 & \text{ on } \Omega,\\ 
        u = 0 & \text{ on }\partial\Omega,
    \end{cases}
\end{equation}
where~$\Omega \subset \mathbb{R}^N$ (with~$N\geq 2$) is a bounded~$C^2$ domain and~$f:[0,+\infty)\rightarrow [0,+\infty)$ is locally Lipschitz continuous in~$[0,+\infty)$, and~$f > \phi_0$ for some positive constant~$\phi_0$. 
We assume that at least one of the following conditions holds true:
\begin{enumerate}[label=(\alph*)]
     \item \label{Assumption 1 on f} $p\in(1,N)$ and there exists a nonincreasing function~$\phi:[0,+\infty)\rightarrow [0,+\infty)$ such that~$\phi\leq f \leq \frac{N}{N-p}\phi$;
     \item \label{Assumption 2 on f} $p\geq N$.
 \end{enumerate}
Our main goal is to estimate ``how far'' a~$C^{1,\alpha}(\overline{\Omega})$ weak solution of~\eqref{Problem} is from being radially symmetric and nonincreasing in terms of a deficit that measures how distant the domain~$\Omega$ is from a ball.  

More specifically, we define the deficit~$\eps$ by 
\begin{equation}\label{Definition of eps}
    \eps:=D(\Omega)+  \inf_{x_0\in \mathbb{R}^N}\int_{\partial \Omega}\left| \frac{N|\Omega|}{\mathcal{H}^{N-1}(\partial \Omega)}- \langle x-x_0, \nu\rangle\right|\,d\mathcal{H}^{N-1}(x),
\end{equation}
where~$\nu$ denotes the outward unit normal to~$\partial \Omega$ and~$D(\Omega)$ denotes the isoperimetric deficit of~$\Omega$ (see Section~\ref{Sec2} below for details).

Note that~$\eps$ is a nonnegative quantity which provides a good measure of how distant~$\Omega$ is from being a ball, since it is the sum of the isoperimetric deficit of~$\Omega$ and a term that measures how distant the outward unit normal of~$\partial \Omega$ is from the unit outward normal of a ball with measure equal to~$\Omega$. 

Furthermore, relations between~$D(\Omega)$ and the second summand in~\eqref{Definition of eps} are present in the literature (see, for example,~\cite{MR2672283}). \medskip

When~$\Omega$ is a ball, the solution~$u$ of~\eqref{Problem} is not only radially symmetric but also radially
nonincreasing (see~\cite{serra2013radial}) and, as such, it coincides with its Schwarz rearrangement~$u^*$, which is defined as 
\begin{equation}\label{Schwarz symmetral of u}
    u^*(x):=\sup\left\{t>0\;{\mbox{ s.t. }}\; |\{u>t\}|\leq \omega_N |x|^N\right\}.
\end{equation}
Therefore, the function~$u^*$ is a natural candidate to estimate how far~$u$ is from being radially symmetric and radially nonincreasing. 
Indeed, our main result is the following:

\begin{theorem}\label{Theorem 1.1} Let~$\Omega\subset \mathbb{R}^N$ be a bounded~$C^2$ domain and let~$f$ be a locally Lipschitz function such that~$f>\phi_0$ for some positive constant~$\phi_0$. 

Let~$u \in C^{1,\alpha}(\overline{\Omega})$ be a weak solution of~\eqref{Problem}. 

Assume that either~\ref{Assumption 1 on f} or~\ref{Assumption 2 on f} hold.

Then, there exist positive constants
$$C:=C(N,p,|\Omega|,\mathcal{H}^{N-1}(\partial \Omega),\phi_0,\|u\|_{L^{\infty}(\Omega)}, \|f\|_{W^{1,\infty}([0,M])}, \mathcal{M}_{0}^{-}) \qquad \text{and} \qquad \theta_1:=\theta_{1}(N,p) ,$$
where
\begin{equation}\label{M 0 -}
\mathcal{M}_{0}^{-}:=\displaystyle\max_{x\in\partial \Omega}\{-H(x),0\} ,
\end{equation}
such that\footnote{We agree to extend $u$ to the rest of $\mathbb{R}^N$ by setting $u \equiv 0$ in $\mathbb{R}^N \setminus \Omega$.}
    \begin{equation}\label{Quantitative stability}
        \min_{x_0 \in \mathbb{R}^N}\int_{\mathbb{R}^N}\left|u(x)-u^*(x+x_0)\right|\,dx \leq C\eps^{\theta_1} .
    \end{equation}
    
    Here, $H(x)$ denotes the mean curvature of~$\partial \Omega$ at~$x$ and~$\eps$ is the deficit defined in~\eqref{Definition of eps}.
\end{theorem}

\begin{remark} {\rm
We stress that the dependence of the constant~$C$ in Theorem~\ref{Theorem 1.1} on the domain only concerns the volume of~$\Omega$, its perimeter, and a bound on the negative part of the mean curvature. In particular it does not involve the regularity of the domain. 

In addition, for mean convex domains (i.e., domains with nonnegative mean curvature), the dependence on~$\mathcal{M}_{0}^{-}$ can be dropped, as we have~$\mathcal{M}_{0}^{-}=0$.}
\end{remark}

The proof of Theorem~\ref{Theorem 1.1} relies on the following two results (namely
Theorems~\ref{Quantitative estimate of Mu} and~\ref{Theorem 1.4})
which provide an estimate on the size of the set~$\{x \in \Omega\,{\mbox{ s.t. }}\,|\nabla u(x)|\leq \sigma\}$ and an integral identity from which one is able to deduce the radial symmetry of solutions to~\eqref{Problem} when~$\Omega$ is a ball, 
respectively. 

The first result is a consequence of summability properties of the gradient of a weak solution of~\eqref{Problem} presented in~\cite{MR2096703} together with a careful analysis of the proofs of said properties. The second result follows from a careful analysis of the arguments in~\cite{serra2013radial}. 

\begin{theorem}\label{Quantitative estimate of Mu} 
Let~$\Omega\subset \mathbb{R}^N$ be a bounded~$C^2$ domain and let~$f$ be a locally Lipschitz function such that~$f>\phi_0$ for some positive constant~$\phi_0$. 

Let~$u \in C^{1,\alpha}(\overline{\Omega})$ be a weak solution of~\eqref{Problem}, with~$p\in(1,+\infty)$.

Then, there exist positive constants
$$
C':=C'(N,p,|\Omega|,\mathcal{H}^{N-1}(\partial \Omega),\phi_0,\|u\|_{L^{\infty}(\Omega)}, \|f\|_{W^{1,\infty}([0,M])}, \mathcal{M}_{0}^{-}) \qquad  \text{and} \qquad  \theta_2:=\theta_2(p)
$$
such that, for all~$\sigma \geq 0$,
\begin{equation}\label{Estimate of Mu}
 \big|\{x \in \Omega\;{\mbox{ s.t. }}\; |\nabla u(x)|\leq \sigma\}\big|\leq C' \sigma^{\theta_2}.
\end{equation}
In particular, the set of critical points of~$u$ has zero $N-$dimensional Lebesgue measure.
\end{theorem}

In what follows, we denote by~$c_N$ the sharp isoperimetric constant, i.e.,
\begin{equation*}
c_N := \frac{ \mathcal{H}^{N-1}( \partial B_1)}{| B_1 |^{\frac{N-1}{N}}} ,
\end{equation*}
where~$B_1$ denotes the unit ball in~$\mathbb{R}^N$ centered at the origin.

\begin{theorem}\label{Theorem 1.4}
Let~$\Omega\subset \mathbb{R}^N$ be a bounded~$C^2$ domain, $f$ be a nonnegative locally Lipschitz\footnote{Here locally bounded is enough.} function and~$u \in C^{1,\alpha}(\overline{\Omega})$ be a weak solution of~\eqref{Problem}. 

Assume that either~\ref{Assumption 1 on f} or~\ref{Assumption 2 on f} hold and let~$x_0 \in \mathbb{R}^N$.

Set
$$
M:= \max_{\overline{\Omega}} u , \qquad  \mu(t):= \left|\left\{x \in \mathbb{R}^N\;{\mbox{ s.t. }}\; u(x)>t\right\}\right|, \qquad{\mbox{and}}\qquad
I(t):=\int_{\{u>t\}} f(u(x))\, dx .
$$ 

Then, the following identity holds
\begin{equation}\label{Identity}\begin{split}&
\int_0^M W(t)\big(D_1(t)+D_2(t)\big)\,dt
+\frac{c_N^{p'}|\Omega|}{\mathcal{H}^{N-1}(\partial\Omega)}D_3\\&\qquad\qquad
= \left(\int_{\Omega} f(u(x))\,dx\right)^{p'} \frac{c_N^{p'}|\Omega|}{\left(\mathcal{H}^{N-1}(\partial\Omega)\right)^{p'}}D_4 + \frac{c_N^{p'}}{N}D_5(x_0),\end{split}
\end{equation}
where
\begin{eqnarray}\label{W}
    W(t) &:= &p'\mu(t)^{\frac{p-N}{N(p-1)}}f(t) + \frac{p-N}{N(p-1)} I(t) \mu(t)^{\frac{p-pN}{N(p-1)}},\\
\label{D1}
    D_1(t)& :=& \left(\int_{\{u=t\}}|\nabla u|^{p-1}\,d\mathcal{H}^{N-1}\right)^{\frac{1}{p-1}} \left(\int_{\{u=t\}} |\nabla u|^{-1}\,d \mathcal{H}^{N-1}\right)\\&&\qquad -\Big( \mathcal{H}^{N-1}(\{u = t\})\Big)^{p'},
\\
\label{D2}
D_2(t) &:=& \Big(\mathcal{H}^{N-1}(\{u = t\})\Big)^{p'} -
\Big(c_N \mu(t)^{(N-1)/N}\Big)^{p'},
\\
\label{D3}
    D_3 &:=& \int_{\partial\Omega}   \left|\frac{\partial u}{\partial \nu}\right|^p\,d\mathcal{H}^{N-1} - \frac{\left(\displaystyle\int_{\partial\Omega}  \left|\frac{\partial u}{\partial \nu}\right|^{p-1}\,d\mathcal{H}^{N-1}\right)^{p'} }{(\mathcal{H}^{N-1}\left(\partial\Omega\right))^{1/(p-1)}}, 
\\
\label{D45}
D_4 &:=& \frac{\left(\mathcal{H}^{N-1}\left(\partial\Omega\right)\right)^{p'}}{(c_N |\Omega|^{1/N'}))^{p'}}- 1, \\
  {\mbox{and}}\quad  D_5(x_0) &:=&\int_{\partial\Omega} \big(
  C_{\Omega}-\langle x-x_0,\nu\rangle \big)\left|\frac{\partial u}{\partial \nu}\right|^p\,d\mathcal{H}^{N-1},
\end{eqnarray}
where
$$
C_{\Omega}:=\frac{ N |\Omega|}{\mathcal{H}^{N-1}\left(\partial\Omega\right)}
$$
and~$p' = \frac{p}{p-1}$ denotes the H\"older conjugate exponent of~$p$. 
\end{theorem}

\begin{remark} {\rm
The integral identity in~\eqref{Identity}, in spite of its seemingly technical flavor,
is important because it
provides the starting point of our quantitative analysis. As a matter of fact, the radial symmetry of solutions of~\eqref{Problem} when~$\Omega$ is a ball follows immediately from~\eqref{Identity} (see Remark~\ref{Implication of the Identity}). }
\end{remark}

The rest of the paper is organized as follows. In Section~\ref{Sec2} we present the notation and background results
that we use throughout the paper. In Section~\ref{SEc3} we prove Theorem~\ref{Theorem 1.4}. In Section~\ref{SEc4} we provide an explicit gradient bound (via $P$-function approach) that will be useful in the level set analysis of a solution of~\eqref{Problem} performed in Section~\ref{Estimates on level sets}.
In Section~\ref{Section 5} we prove Theorem~\ref{Quantitative estimate of Mu}. We finalize the paper with the proof of Theorem~\ref{Theorem 1.1} in Section~\ref{sec7}. 


\section{Notation and known facts from the literature}\label{Sec2}

In this section, we present the notation and the results used throughout the paper.

Given a set~$A\subset \mathbb{R}^{N}$, we denote by~$|A|$ the $N-$dimensional Lebesgue measure of~$A$ and by~$\mathcal{H}^{N-1}(A)$
its $(N-1)-$dimensional Hausdorff measure.  

Given a finite positive measure space~$(X,\varpi)$ and a $\varpi-$measurable function~$f: X\rightarrow \mathbb{R}$, we denote by~$\fint_X f\,d\varpi$ the average of~$f$ over~$X$
according to the measure~$\varpi$, i.e., 
$$
\fint_X f\,d\varpi:= \frac{1}{\varpi(X)}\int_X f\,d\varpi.
$$

Throughout the paper, we will make extensive use of the coarea formula (see~\cite{MR0257325}), more specifically of the coarea formula for Sobolev functions~$u\in W^{1,1}_{\text{loc}}(\mathbb{R}^N)$, that states that, for any measurable function~$v:\mathbb{R}^N\rightarrow [0,+\infty)$,
\begin{equation}\label{Coarea formula 1}
    \int_{\mathbb{R}^N} v(x) |\nabla u(x)| \,dx = \int_{\mathbb{R}}\int_{\{u = t\}} v(x)\,d\mathcal{H}^{N-1}\,dt.
\end{equation}
In particular, for any bounded open set~$\Omega \subset \mathbb{R}^N$,
\begin{equation*}
    \begin{split}
        |\Omega| &= \big|\{x \in \Omega\;{\mbox{ s.t. }}\; |\nabla u(x)| = 0\}\big| + \big|\{x\in \Omega\;{\mbox{ s.t. }}\; \nabla u(x) \neq 0\}\big|\\
        & = \big|\{x \in \Omega\;{\mbox{ s.t. }}\; |\nabla u(x)| = 0\}\big|+ 
        \int_{\Omega\cap \{\nabla u \neq 0\}}\frac{|\nabla u(x)|}{|\nabla u(x)|}\,dx\\
        & = \big|\{x \in \Omega\;{\mbox{ s.t. }}\; |\nabla u(x)| = 0\}\big|+ \int_{\mathbb{R}}\int_{\{u = t\}\cap \Omega} \frac{\chi_{\{\nabla u \neq 0\}}(x)}{|\nabla u(x)|}\,d\mathcal{H}^{N-1}\,dt.
    \end{split}
\end{equation*}

Given a function~$u: \mathbb{R}^N\rightarrow [0,+\infty)$ such that~$|\{x \in \mathbb{R}^N\,{\mbox{ s.t. }}\, u(x)>0\}|<+\infty$, we define the distribution function~$\mu: [0,+\infty) \rightarrow [0,+\infty)$ of~$u$ by 
\begin{equation}\label{Distribution function of u}
    \mu(t):= \left|\left\{x \in \mathbb{R}^N\;{\mbox{ s.t. }}\; u(x)>t\right\}\right|.
\end{equation}
The Schwarz symmetric rearrangement~$u^*$ of~$u$ is defined, for any~$x \in \mathbb{R}^N$, by 
\begin{equation}\label{Schwarz symmetral of u}
    u^*(x):=\sup\left\{t>0\;{\mbox{ s.t. }}\; \mu(t)\leq \omega_N |x|^N\right\}.
\end{equation}
Note that, by definition, the function~$u^*$ is radially symmetric and the distribution function~$\mu^*(t)$ of~$u^*$ coincides with the distribution function of~$u$, i.e., $\mu(t) = \mu^*(t)$.

Furthermore, if~$u \in W^{1,p}\left(\mathbb{R}^N\right)$, then~$u^* \in W^{1,p}\left(\mathbb{R}^N\right)$, where~$W^{1,p}\left(\mathbb{R}^N\right)$ denotes the Sobolev Space of weakly differentiable measurable functions that are $p-$integrable and whose (weak) derivative is $p-$integrable, and the Pólya-Szegő inequality states that
\begin{equation*}
    \int_{\mathbb{R}^N}|\nabla u^*(x)|^p \,dx \leq \int_{\mathbb{R}^N}|\nabla u(x)|^p\,dx.
\end{equation*}
Quantitative versions of this inequality for a rather general class of functions in~$W^{1,p}\left(\mathbb{R}^N\right)$ have been studied in~\cite{CianchiEspositoFuscoTrombetti+2008+153+189} and for $\log-$concave and $\alpha-$concave functions in~\cite{MR3250365}.

A consequence of Lemmata~$2.4$ and~$2.6$ in~\cite{Cianchi2002FunctionsOB} is that 
$$
\frac{d}{dt}\left|\left\{x \in \mathbb{R}^N\;{\mbox{ s.t. }}\; u^*(x)>t\,{\mbox{ and }}\, |\nabla u^*(x)|=0\right\}\right| = 0, 
$$
and thus the coarea formula shows that, for a.e.~$t \geq 0$,
\begin{equation}\label{Derivative of the distribution function 1}
\begin{split}
   \frac{d}{dt}\mu(t) &= \frac{d}{dt}\left(
   \big|\big\{x \in \{u>t\}\;{\mbox{ s.t. }}\; |\nabla u^*(x)| = 0\big\}\big|+ \int_{\mathbb{R}}\int_{\{u^* = t\}} \frac{\chi_{\{\tau>t\}}(\tau)}{|\nabla u^*(x)|}\,d\mathcal{H}^{N-1}\,d\tau \right)\\
   &=-\int_{\{u^*=t\}} \frac{d\mathcal{H}^{N-1}}{|\nabla u^*(x)|}.
   \end{split}
\end{equation}

The isoperimetric inequality and its quantitative counterpart will also play an important role throughout this paper,
therefore we present here the basic definitions and results that we will use. 

Given an open subset~$\Omega$ of~$\mathbb{R}^N$ and a measurable subset~$E$ of~$\mathbb{R}^N$, the perimeter of~$E$ in~$\Omega$ is given by 
\begin{equation}\label{Perimeter of a set}
    P(E;\Omega) := \sup\left\{\int_{E} \div \varphi \,dx \,:\, \varphi\in C^{\infty}_c(\Omega;\mathbb{R}^N),\, \|\varphi\|_{L^{\infty}}\leq 1\right\}.
\end{equation}
When~$\Omega = \mathbb{R}^N$ we denote the perimeter of~$E$ in~$\mathbb{R}^N$ by~$P(E)$.

We say that~$E$ is a set of finite perimeter if~$P(E)<+ \infty$. In this case, the isoperimetric inequality gives that
\begin{equation}\label{Isoperimetric inequality 1}
    P(E)\geq P(B_R),
\end{equation}
where~$B_R$ is a ball with the same $N-$dimensional Lebesgue measure of~$E$, and we have equality if and only if~$E$ is equivalent to a ball.

The isoperimetric inequality can also be written in the following equivalent form:
\begin{equation}\label{Isoperimetric inequality 2}
    P(E) \geq c_N \left|E\right|^{\frac{N-1}{N}}, \qquad \text{where }\; c_N := \frac{P(B_1)}{| B_1 |^{\frac{N-1}{N}}}.
\end{equation}
Recall that when~$E$ is of class~$C^1$, we have that~$P(E)= \mathcal{H}^{N-1}(\partial E)$. 

Throughout the years, many authors have studied quantitative versions of the isoperimetric inequality: one that will be useful in the present paper is the one put forth
in~\cite{MR2456887}, where  N. Fusco, F. Maggi and A. Pratelli showed that there exists a constant~$\gamma(N)$ such that when~$E$ is a set of finite perimeter, then 
\begin{equation}\label{Quantitative isoperimetric inequality}
   \displaystyle\min_{x\in \mathbb{R}^N}\left\{\frac{\left|E \Delta B_R(x)\right|^2}{|B_R|^2}\right\}\leq \gamma(N) D(E),
\end{equation}
where~$B_R(x)$ denotes the ball of radius~$R$ and center~$x$, $R$ is such that~$|B_R|=|E|$, $E\Delta B_R(x):= \left( E\setminus B_R(x) \right) \cup \left( B_R(x)\setminus E \right)$ denotes the symmetric difference of the sets~$E$ and~$B_R(x)$, and~$D(E)$ denotes the isoperimetric deficit of the set~$E$, that is,
\begin{equation}\label{e8239ryjkefg239yr3829}
D(E) := \frac{P(E)-P(B_R)}{P(B_R)} .
\end{equation}
For an extensive exposition on the isoperimetric inequality and various quantitative versions of the isoperimetric inequality such as the one above see~\cite{Fusco_2015}.

To make the quantitative study of the level sets mentioned in the Introduction we will also use a quantitative version of H\"older's inequality that can be found in Proposition~2.3 of~\cite{CianchiEspositoFuscoTrombetti+2008+153+189}. This result states that, for every~$p\in(1,+\infty)$, there exists a positive
constant~$C_p$ such that, for every finite positive measure space~$(X,\varpi)$ and
every~$\varpi-$measurable function~$f: X\rightarrow [0,+\infty)$ 
such that~$f>0$ $\varpi-$a.e. in~$X$, $\int_X f^{p-1}\,d\varpi <+\infty$ and~$\int_X f^{-1}\,d\varpi<+\infty$, one has that
\begin{equation}\label{Quantitative Holder}
\fint_X\left|\frac{1}{f}-\overline{f}\right|\,d\varpi \leq C_p \overline{f}\begin{cases} D_{f}^{\frac1{p}},\,&\quad\text{if }\, p\geq 2\\
\left(D_{f}^{p-1}+D_{f}^{\frac{1}{p-1}}\right)^{\frac1{p}}, \,&\quad\text{if }\, 1<p< 2,
\end{cases}
\end{equation}
where
\begin{eqnarray*} 
\overline{f}&:=& \left(\fint_X\frac{1}{f}\, d\varpi \right)^{\frac1{p'}}\left(\fint_X f^{p-1}\, d\varpi \right)^{-\frac{1}{p-1}}\\
{\mbox{and }}\quad 
D_{f}&:= &\left(\fint_X \frac{1}{f}\, d\varpi \right)^{p-1}\left(\fint_X f^{p-1}\, d\varpi \right)-1.
\end{eqnarray*}
Here above and in the rest of the paper~$p' = \frac{p}{p-1}$ denotes the H\"older conjugate exponent of~$p$. 

Given a subset~$\Omega$ of~$\mathbb{R}^N$ we define, for all~$\delta \geq 0$,
$$
\Omega_{\delta}:=\big\{x\in \mathbb{R}^N\;{\mbox{ s.t. }}\; \text{dist}(x,\Omega)\leq \delta\big\}.
$$
Lemma~2.3 in~\cite{poggesi2024bubbling} shows that
when~$\Omega$ is a bounded domain of class~$C^2$, the volume of the portion of the $\delta-$neighbourhood of~$\partial \Omega$ contained in~$\Omega$, denoted by~$\mathcal{V}(\delta)$, has the following upper bound
\begin{equation}\label{Size of the tube}
    \mathcal{V}(\delta) := \big|\big\{x\in \Omega\;{\mbox{ s.t. }}\; \text{dist}(x,\partial \Omega)\leq \delta\big\}\big|\leq \left(1+\delta\mathcal{M}_0^{-}\right)^{N-1}\mathcal{H}^{N-1}\left(\partial\Omega \right) \delta,
\end{equation}
where~$\mathcal{M}_{0}^{-}$ is as defined in~\eqref{M 0 -}.

Now take~$\Omega\subset \mathbb{R}^N$ to be a bounded~$C^2$ domain and~$u \in C^{1,\alpha}(\overline{\Omega})$ a (weak) solution of~\eqref{Problem}. Let
$
M:= \| u \|_{L^{\infty}(\Omega)} .
$
Since~$f(u)>\phi_0$ and~$|\{0<u<M\}\cap \{\nabla u = 0\}|= 0$, it follows from the proof of Lemma~4 in~\cite{serra2013radial} and the coarea formula that,
for a.e.~$t\in (0,M)$,
\begin{equation}\label{Derivative of the distribution function 2}
    \frac{d}{dt}\mu(t) = -\int_{\{u =t\}} \frac{d\mathcal{H}^{N-1}}{|\nabla u|}.
\end{equation}
Therefore, for a.e.~$t \in (0,M)$,
\begin{equation}\label{Equality of the derivative of dist fct}
    \int_{\{u =t\}} \frac{d\mathcal{H}^{N-1}}{|\nabla u|}= \int_{\{u^* =t\}} \frac{d\mathcal{H}^{N-1}}{|\nabla u^*|}.
\end{equation}
Furthermore, by standard elliptic regularity theory we know that $u \in C^2(\Omega\setminus \{x \in \Omega:\,\nabla u =0\})$ (see~\cite{MR0709038, MR0737190}).
Also, as in~\cite{serra2013radial}, we define the two additional distribution type functions~$I(t)$ and~$K(t)$ for~$t \in (0,M)$ by 
\begin{equation}\label{Definition of I and K}
I(t):=\int_{\{u>t\}} f(u(x))\, dx\qquad{\mbox{and}}\qquad
K(t) := I(t)^{p'} \mu(t)^{\frac{p-N}{N(p-1)}}.
\end{equation}
{F}rom Lemma~4 in~\cite{serra2013radial} (under the assumption that~$u$ is a solution of~\eqref{Problem}) we know that the functions~$I$ and~$K$ are a.e. differentiable, with
 \begin{equation}\label{Derivative of I}
     I'(t) = f(t)\mu'(t)
 \end{equation}
and
 \begin{equation}\label{Derivative of K}
    K'(t) = \left(p'I(t)^{\frac{1}{p-1}} \mu(t)^{\frac{p-N}{N(p-1)}}f(t) + \frac{p-N}{N(p-1)} I(t)^{p'} \mu(t)^{\frac{p-Np}{N(p-1)}} \right) \mu'(t),
\end{equation}
where both equalities hold for a.e.~$t\in (0,M)$.

{F}rom Lemma~5 in~\cite{serra2013radial} we also know that the following Gauss-Green identity holds true
\begin{equation}\label{I}
    I(t) = \int_{ \{u=t\} } |\nabla u|^{p-1} d\mathcal{H}^{N-1},
\end{equation}as well as 
that the isoperimetric inequality holds for the level sets of the solution~$u$, namely
\begin{equation}\label{Isoperimetric inequality for level sets}
    \mathcal{H}^{N-1}(\{u=t\}) \geq c_N \mu(t)^{\frac{N-1}{N}},
\end{equation}
where~$c_N$ is the sharp isoperimetric constant in~$\mathbb{R}^N$. 

We finish this section by recalling the Pohozaev identity (see~\cite{A-general-variational-identity}):
\begin{equation}\label{Pohozaev}
\int_{\Omega}\left( N F(u(x)) +\frac{N-p}{p}u(x)f(u(x))\right)\,dx 
=\frac{1}{p'} \int_{\partial \Omega}\left|\frac{\partial u}{\partial \nu}(x)\right|^{p}\langle x-x_0, \nu\rangle\,d\mathcal{H}^{N-1},
\end{equation}
where~$F(u) := \displaystyle\int_{0}^{u} f(s)\,ds$ and~$x_0 \in \mathbb{R}^N$.

\section{A fundamental identity and proof of Theorem~\ref{Theorem 1.4}}\label{SEc3}
Following the arguments presented in~\cite{serra2013radial} we can deduce an identity that contains (almost) all the information about the level and super-level sets of the solution of~\eqref{Problem}, as well as the geometry of the domain~$\Omega$ and how they relate to each other.
This identity will be crucial to detect how close the solution of~\eqref{Problem} is to be radially symmetric via a level and super-level set analysis when the set~$\Omega$ is not a ball (besides, it also provides an alternative proof of radial symmetry of solutions of~\eqref{Problem} when~$\Omega$ is a ball).

\begin{proof}[Proof of Theorem~\ref{Theorem 1.4}]
{F}rom the properties of~$I$, $K$ and~$\mu$ contained in~\eqref{Derivative of the distribution function 2},
\eqref{Derivative of K} and~\eqref{I}, it follows that
\begin{equation}\label{K(0-)}
    \begin{split}&
        K(0^{-})\\&= \int_{0}^M -K'(t)\,dt\\
        &= -\int_0^M  \left(p'I(t)^{\frac{1}{p-1}} \mu(t)^{\frac{p-N}{N(p-1)}}f(t) + \frac{p-N}{N(p-1)} I(t)^{p'} \mu(t)^{\frac{p-Np}{N(p-1)}} \right) \mu'(t)\,dt\\
        &= \int_0^M  \left(p'I(t)^{\frac{1}{p-1}} \mu(t)^{\frac{p-N}{N(p-1)}}f(t) + \frac{p-N}{N(p-1)} I(t)^{p'} \mu(t)^{\frac{p-Np}{N(p-1)}} \right) \left(\int_{\{u=t\}} |\nabla u|^{-1}\,d\mathcal{H}^{N-1}\right)\,dt\\
        &= \int_0^M \left(p' \mu(t)^{\frac{p-N}{N(p-1)}}f(t) + \frac{p-N}{N(p-1)} I(t) \mu(t)^{\frac{p-Np}{N(p-1)}} \right) I(t)^{\frac{1}{p-1}} \left(\int_{\{u=t\}} |\nabla u|^{-1}\,d\mathcal{H}^{N-1}\right)\, dt\\
        &= \int_0^M W(t)\left(\int_{\{u = t\}} |\nabla u| d \mathcal{H}^{N-1}\right)^{\frac{1}{p-1}} \left(\int_{\{u = t\}} \frac{d \mathcal{H}^{N-1}}{|\nabla u|}\right)\,dt.
    \end{split}
\end{equation}
Using the definitions of~$W(t)$, $D_1(t)$ and~$D_2(t)$,
contained, respectively, in~\eqref{W}, \eqref{D1} and~\eqref{D2},
we see that
\begin{equation}\label{est 2}
    K(0^-) = \int_0^M W(t)\big(D_1(t)+ D_2(t)\big)\,dt + c_N^{p'}\int_0^M \left(p' f(t)\mu(t) + \frac{p-n}{n(p-1)}I(t)\right)\,dt.
\end{equation}Also, using the Pohozaev's identity,
\begin{equation*}
\begin{split}
     c_N^{p'}\int_0^M p' f(t)\mu(t) + \frac{p-N}{N(p-1)}I(t)\,dt &= \frac{c_N^{p'}}{N} \int_{\partial \Omega} \langle x-x_0, \nu\rangle \left|\frac{\partial u}{\partial \nu}\right|^p\,d\mathcal{H}^{N-1},
\end{split}
\end{equation*}
for any~$x_0 \in \mathbb{R}^N$.

Since
\begin{equation*}
    K(0^-) = \left(\int_{\Omega} f(u)\,dx\right)^{p'}|\{u>t\}|^{\frac{p-N}{N(p-1)}},
\end{equation*}
and recalling~\eqref{est 2} and the definitions of the deficits~$D_3$, $D_4$
and~$D_5(x_0)$, we find that
\begin{equation*}
    \begin{split}
       \int_0^M W(t)\big(D_1(t)+ D_2(t)\big)\,dt &= \left(\int_{\Omega} f(u)\,dx\right)^{p'}|\Omega|^{\frac{p-N}{N(p-1)}} - \frac{c_N^{p'}}{N} \int_{\partial \Omega} \langle x-x_0, \nu\rangle \left|\frac{\partial u}{\partial \nu}\right|^p \,d\mathcal{H}^{N-1} \\
        &= \frac{c_N^{p'}|\Omega|}{\mathcal{H}^{N-1}(\partial\Omega)^{p'}}\left(\int_{\Omega} f(u)\,dx\right)^{p'} D_4 + \frac{c_N^{p'}}{N}D_5(x_0)- \frac{c_N^{p'}|\Omega|}{\mathcal{H}^{N-1}\left(\partial\Omega\right)}D_3.
    \end{split}
\end{equation*}
The proof of Theorem~\ref{Theorem 1.4} is thereby complete.
\end{proof}

\begin{remark}\label{Implication of the Identity} 
{\rm
    We point out that the deficits~$D_1(t)$ and~$D_3$ are nonnegative due to H\"older's inequality, and that the deficits~$D_2(t)$ and~$D_4$ are also nonnegative due to the isoperimetric inequality.

Also note that when~$\Omega$ is a ball, the deficit~$D_4$ is zero and there exists~$x_0\in\mathbb{R}^N$ for which~$D_5(x_0)$ is also zero. This, together with~\eqref{Identity} and the fact that the deficits~$D_1(t)$, $D_2(t)$ and~$D_3$ are all nonnegative, implies that~$D_1(t)$, $D_2(t)$ and~$D_3$ are zero for a.e.~$t>0$ since~$W(t)\geq 0$ (see Lemma~\ref{Positivity of W} below). In particular, since the isoperimetric deficits of the super level sets~$\{u>t\}$ of~$u$ are zero, the level sets of~$u$ are spheres. 

Moreover, since~$D_1(t)$ is zero for a.e.~$t>0$, we also see that~$|\nabla u|$ is constant on the level sets~$\{u=t\}$ for a.e.~$t>0$. This implies that the solution of~\eqref{Problem} is radially symmetric (see, e.g., 
\cite{serra2013radial}).   }
\end{remark}

\section{A gradient bound}\label{SEc4}

This section is devoted to obtaining a gradient bound which, in the next section, will be used to justify and compute the deficit that we use to measure how ``distant'' the domain~$\Omega$ is from a ball.

\begin{lemma}
Let~$\Omega\subset \mathbb{R}^N$ be a bounded~$C^2$ domain, $f \ge 0$ be a locally Lipschitz function and~$u \in C^{1,\alpha}(\overline{\Omega})$ be a weak solution of~\eqref{Problem}.

Then,
\begin{equation}\label{Gradient bound}
\|\nabla u \|_{L^{\infty}(\Omega)} \leq \overline{C} ,
\end{equation}
where~$\overline{C}$ is a constant  depending only on~$p$, $N$, $M$, $\|f\|_{L^{\infty}([0,M])}$, and~$\mathcal{M}_{0}^{-}$. 

Here, 
$$
M:= \| u \|_{L^\infty (\Omega)} \qquad \text{and} \qquad \mathcal{M}_{0}^{-}:=\displaystyle\max_{x\in\partial \Omega}\{-H(x),0\}.
$$
\end{lemma}

\begin{proof}
We follow the outline in~\cite{MR4476237} and refer to~\cite{MR4476237} for additional details.
We define 
$$
P := \frac{1}{p'}|\nabla u|^p - \int_{u}^{M}f(s)\,ds + \beta (u-M),
$$
where~$\beta$ is to be chosen conveniently later on.

Taking the gradient of equation~\eqref{Problem} we see that 
\begin{equation}\label{Gradient of the equation}
\begin{split}
- f'(u) \nabla u =\,&\nabla(\Delta u)|\nabla u|^{p-2} + (p-2)\Delta u |\nabla u|^{p-3}\nabla (|\nabla u|)\\
&\qquad+ (p-2)(p-3)|\nabla u |^{p-3} \nabla (|\nabla u |) \langle \nabla (|\nabla u|),\nabla u \rangle\\
&\qquad+ (p-2)|\nabla u |^{p-3} \nabla^2 (|\nabla u|) \nabla u +(p-2)|\nabla u|^{p-3} \nabla^2 u \nabla (|\nabla u|).
\end{split}
\end{equation}
Thus, taking the the inner product of this equation with~$\nabla u$,
\begin{equation*}
\begin{split}
-f'(u)|\nabla u |^2 =\,& |\nabla u|^{p-2}\langle \nabla(\Delta u),\nabla u \rangle + (p-2)\Delta u  |\nabla u|^{p-3}\langle \nabla (|\nabla u|),\nabla u \rangle\\
&\qquad+ (p-2)(p-3)|\nabla u |^{p-3}| \langle \nabla (|\nabla u|),\nabla u \rangle|^2\\
&\qquad+(p-2)|\nabla u |^{p-3} \langle\nabla^2 (|\nabla u|) \nabla u, \nabla u \rangle+  (p-2)|\nabla u|^{p-3} \langle\nabla^2 u \nabla (|\nabla u|),\nabla u \rangle
\end{split}
\end{equation*}
and therefore
\begin{equation}\label{third order terms}
\begin{split}
&f'(u)|\nabla u|^2 + |\nabla u|^{p-2}\langle \nabla(\Delta u),\nabla u \rangle  + (p-2)|\nabla u|^{p-3} \langle\nabla^2 u \nabla (|\nabla u|),\nabla u \rangle\\
=\,&  -(p-2)\Delta u  |\nabla u|^{p-3}\langle \nabla (|\nabla u|),\nabla u \rangle- (p-2)(p-3)|\nabla u |^{p-4}| \langle \nabla (|\nabla u|),\nabla u \rangle|^2\\
&\qquad- (p-2)|\nabla u|^{p-3} \langle\nabla^2 u \nabla (|\nabla u|),\nabla u \rangle.
\end{split}
\end{equation}

Additionally, from the definition of the function~$P$ the following identities hold true:
\begin{eqnarray*}
\nabla P& =& (p-1)|\nabla u|^{p-1}\nabla (|\nabla u|)+ (\beta +f(u))\nabla u,
\\
\nabla^2 P &=& (p-1)|\nabla u |^{p-1}\nabla^2(|\nabla u|)+ (p-1)^2|\nabla u|^{p-2}\nabla(|\nabla u|)\otimes \nabla (|\nabla u|)\\&&\qquad+ (\beta+ f(u))\nabla^2 u + f'(u)\nabla u \otimes \nabla u,\\
|\nabla P|^2 &=& (p-1)^2|\nabla u|^{2p-2}|\nabla(|\nabla u|)|^2 +  (\beta +f(u))^2|\nabla u|^2 \\&&
\qquad+ 2(p-1)(\beta+f(u))|\nabla u|^{p-1}\langle \nabla(|\nabla u|),\nabla u\rangle,\\
	\langle \nabla P,\nabla u\rangle &=& (p-1)|\nabla u|^{p-1}\langle \nabla(|\nabla u|),\nabla u\rangle + (\beta+ f(u))|\nabla u|^2
\\
{\mbox{and }}\quad
\left|\langle \nabla P,\nabla u\rangle\right|^2 &=& (p-1)^2 |\nabla u|^{2p-2}|\langle \nabla(|\nabla u|),\nabla u\rangle|^2+ (\beta+ f(u))^2|\nabla u|^4  \\
&&\qquad+ 2(p-1)(\beta+f(u))|\nabla u|^{p+1}\langle \nabla(|\nabla u|),\nabla u\rangle .
\end{eqnarray*}

We now define the additional auxiliary functions~$A: \mathbb{R}^N \rightarrow \mathbb{R}^{N\times N}$ and~$a$, $b$, $c: (0,+\infty)\rightarrow\mathbb{R}$ by
\begin{equation*}\begin{split}&
A(\xi) : =I_N + \frac{p-2}{|\xi|^2} \xi\otimes \xi,\qquad
a(\sigma) := \frac{1}{(p-1)\sigma^p},\\& b(\sigma) := \frac{1-(p-1)^2}{2(p-1)\sigma^{p+2}}, \qquad {\mbox{and}}\qquad
c(\sigma):= \frac{pf(u)+2(p-1)\beta}{\sigma^{p}}.\end{split}
\end{equation*}
Then, we have that
\begin{equation}\label{TBEL}
\begin{split}&
\text{tr}\left(A(\nabla u ) \nabla^2 P\right) \\&\quad= \text{tr}(\nabla^2 P) + \frac{p-2}{|\nabla u|^2}tr\left[\left(\nabla u \otimes \nabla u\right)\nabla^2 P\right]	\\
&\quad= (p-1)|\nabla u|^{p-1}\Delta (|\nabla u|) + (p-1)|\nabla u|^{p-2}|\nabla (|\nabla u|)|^2 +(\beta +f(u))\Delta u + f'(u)|\nabla u|^2\\
&\quad\qquad+ \frac{p-2}{|\nabla u|^2}\sum_{i,l} \partial_i u \partial_l u \left(\nabla^2 P\right)_{li}\\
&\quad= (p-1)|\nabla u|^{p-1}\Delta (|\nabla u|) + (p-1)^2|\nabla u|^{p-2}|\nabla (|\nabla u|)|^2 +(\beta +f(u))\Delta u + f'(u)|\nabla u|^2\\
&\quad\qquad+ (p-2)|\nabla u |^2 f'(u) +  \frac{p-2}{|\nabla u|^2}(\beta +f(u)) \langle \nabla^2 u\nabla u,\nabla u \rangle \\
&\quad\qquad+ (p-2)(p-1)^2|\nabla u|^{p-4}\sum_{i,l}\partial_i u \partial_l u \partial_i (|\nabla u|)\partial_l(|\nabla u|)\\&\quad\qquad+ (p-1)(p-2)|\nabla u|^{p-3} \langle \nabla^2(|\nabla u|)\nabla u,\nabla u\rangle  \\
&\quad= (p-1)\left[|\nabla u|^{p-1}\Delta (|\nabla u|) +|\nabla u|^2 f'(u) +(p-2)|\nabla u|^{p-3} \langle \nabla^2(|\nabla u|)\nabla u,\nabla u\rangle \right]\\
&\quad\qquad+ (\beta + f(u))(\Delta u  + (p-2)|\nabla u|^{-1}\langle  \nabla(|\nabla u|), \nabla u \rangle) + (p-1)^2|\nabla u|^{p-2}|\nabla (|\nabla u|)|^2\\
&\quad\qquad+ (p-1)^2(p-2)|\nabla u|^{p-4}|\langle \nabla (|\nabla u|),\nabla u\rangle |^2.
\end{split}
\end{equation}
Using~\eqref{third order terms} and the following equality
$$
\Delta |\nabla u| =|\nabla u|^{-1}\left(\langle \nabla \Delta u ,\nabla u\rangle + \text{tr}((\nabla^2 u)^2 )- |\nabla (|\nabla u|)|^2\right),
$$
we can eliminate the third order terms in~\eqref{TBEL} to obtain 
\begin{equation}\label{fgtbs}
\begin{split}&
\text{tr}\left(A(\nabla u ) \nabla^2 P\right)\\ &\quad= (p-1)\left[|\nabla u|^{p-2}\text{tr}((\nabla^2 u)^2)-|\nabla u|^{p-2}|\nabla (|\nabla u|)|^2 -(p-2)\Delta u  |\nabla u|^{p-3}\langle \nabla (|\nabla u|),\nabla u \rangle\right] \\
&\quad\qquad-(p-1)\left[(p-2)(p-3)|\nabla u |^{p-4}| \langle \nabla (|\nabla u|),\nabla u \rangle|^2 + (p-2)|\nabla u|^{p-3} \langle\nabla^2 u \nabla (|\nabla u|),\nabla u \rangle\right]\\
&\quad\qquad+ (\beta + f(u))(\Delta u  + (p-2)|\nabla u|^{-1}\langle  \nabla(|\nabla u|), \nabla u \rangle) + (p-1)^2|\nabla u|^{p-2}|\nabla (|\nabla u|)|^2\\
&\quad\qquad+(p-1)^2(p-2)|\nabla u|^{p-4}|\langle \nabla (|\nabla u|),\nabla u\rangle |^2.
\end{split}
\end{equation}
Also, by~\eqref{Problem}, 
$$
|\nabla u|^{p-2}\Delta u = -f(u)- (p-2)|\nabla u|^{p-3}\langle \nabla(|\nabla u|),\nabla u\rangle.
$$
{F}rom this and the fact that 
\begin{equation*}
|\nabla u|^{p-3}\langle \nabla^2 u\nabla (|\nabla u|),\nabla u\rangle = |\nabla u|^{p-2}|\nabla(|\nabla u|)|^2 ,
\end{equation*}
the expression in~\eqref{fgtbs} simplifies to
\begin{equation*}
\begin{split}
\text{tr}\left(A(\nabla u ) \nabla^2 P\right) &= (p-1)|\nabla u|^{p-2}\text{tr}((\nabla^2 u)^2)+p(p-1)(p-2)|\nabla u|^{p-4}|\langle \nabla (|\nabla u|),\nabla u\rangle|^2\\
&\qquad+(p-1)(p-2)f(u)\frac{\langle\nabla (|\nabla u|),\nabla u\rangle}{|\nabla u|}- |\nabla u|^{2-p}(\beta + f(u))f(u) .
\end{split}
\end{equation*}

By the definition of the auxiliary functions~$a$, $b$, and~$c$ we have that
\begin{equation*}
\begin{split}
a(|\nabla u|)|\nabla P|^2&=\frac{1}{(p-1)|\nabla u|^p}\left[ (p-1)^2|\nabla u|^{2(p-1)}|\nabla (|\nabla u|)|^2+ (\beta+f(u))^2|\nabla u|^2\right]\\
&\qquad+\frac{2}{|\nabla u|^p}(\beta+f(u))|\nabla u|^{p-1}\langle \nabla(|\nabla u|),\nabla u\rangle\\
&= (p-1)|\nabla u|^{p-2} |\nabla (|\nabla u|)|^2 + \frac{(\beta+f(u))^2}{p-1}|\nabla u|^{2-p} + 2(\beta+f(u))\frac{\langle \nabla (|\nabla u|),\nabla u\rangle}{|\nabla u |} 
\end{split}
\end{equation*}
and that
\begin{equation*}
\begin{split}
b(|\nabla u|) \left|\langle \nabla P,\nabla u\rangle\right|^2 &=\frac{p(2-p)}{(p-1)|\nabla u|^{p+2}}\left((p-1)^2 |\nabla u|^{2p-2}|\langle \nabla(|\nabla u|),\nabla u\rangle|^2+ (\beta+ f(u))^2|\nabla u|^4  \right)\\
&\qquad+ 2p(2-p)|(\beta+f(u))\frac{\langle \nabla(|\nabla u|),\nabla u\rangle}{|\nabla u|}\\
&= p(2-p)(p-1)|\nabla u|^{p-4}|\langle (|\nabla u|),\nabla u\rangle|^2 + \frac{p(2-p)}{p-1}|\nabla u|^{2-p}(\beta+f(u))^2\\
&\qquad+ 2p(2-p)(\beta+f(u))\frac{\langle \nabla(|\nabla u|),\nabla u\rangle}{|\nabla u|} .
\end{split}
\end{equation*}

Thus, for~$\beta \geq 0$ we obtain
\begin{equation*}
\begin{split}
&\text{tr}(A(\nabla u)\nabla^2 P) - a(|\nabla u|)|\nabla P|^2 + b(|\nabla u|) \left|\langle \nabla P,\nabla u\rangle\right|^2+c(|\nabla u|)\langle \nabla P,\nabla u\rangle\\
= \,&(p-1)|\nabla u |^{p-2}\text{tr}((\nabla^2 u)^2)+ p(p-1)(p-2)|\langle \nabla (|\nabla u|),\nabla u\rangle|^2\\
&\;-(p-1)|\nabla u|^{p-2}|\nabla (|\nabla u|)|^2 - p(p-1)(p-2)|\langle \nabla (|\nabla u|),\nabla u\rangle|^2\\
&\;+\Big((p-1)(p-2)f(u) -2(f(u)+\beta)+2p(2-p)(f(u)+\beta) + (p-1)pf(u)+2(p-1)^2\beta\Big)\frac{\langle \nabla (\nabla u),\nabla u \rangle}{|\nabla u|}\\
&\;+\left(-(f(u)+\beta)f(u)-\frac{(\beta+f(u))^2}{p-1}+\frac{p(2-p)}{p-1}(f(u)+\beta)^2 +(pf(u)+2(p-1)\beta)(\beta+f(u))\right)|\nabla u|^{2-p} \\
=\,&(p-1)|\nabla u |^{p-2}\text{tr}((\nabla^2 u)^2) - (p-1)|\nabla u|^{p-2}|\nabla (|\nabla u|)|^2+ (p-1)\beta(f(u)+\beta) |\nabla u|^{2-p}\\
\geq\,& 0.
\end{split}
\end{equation*}
In view of~\eqref{Problem} we have that 
$$
\Delta u = -(p-2)\frac{\langle \nabla(|\nabla u|),\nabla u \rangle}{|\nabla u|}-f(u)|\nabla u|^{2-p}.
$$
Using the fact that on~$\partial \Omega$
$$
\frac{\langle\nabla(|\nabla u|),\nabla u \rangle}{|\nabla u|} = u_{\nu \nu}= \Delta u +(N-1)H(x)|\nabla u|,
$$
we conclude that
$$
(p-1)\frac{\langle \nabla (|\nabla u|),\nabla u \rangle}{|\nabla u |}= -f(u)|\nabla u|^{2-p}+ (N-1)H(x)|\nabla u|.
$$

Furthermore, the inner unit normal to~$\partial \Omega$ at~$x$ is given by~$
-\nu=
|\nabla u|^{-1}\nabla u$. Therefore, if~$\beta = (N-1)\mathcal{M}_{0}^{-}\|\nabla u\|_{L^{\infty}(\partial \Omega)}^{p-1}+1$, we find that
\begin{equation*}
\begin{split}
-\frac{\partial P}{\partial \nu} &= \frac{1}{|\nabla u|}\langle \nabla P,\nabla u \rangle = (p-1)|\nabla u|^{p-2}\langle \nabla(|\nabla u|),\nabla u\rangle +|\nabla u| (f(u)+\beta)\\
&= |\nabla u|f(u) + |\nabla u|\beta -f(u)|\nabla u| + (N-1)H(x)|\nabla u|^p\\
&=\left(\beta + (N-1)H(x)|\nabla u|^{p-1}\right)|\nabla u|\\&\geq |\nabla u|>0,
\end{split}
\end{equation*}
where the last inequality follows from the Hopf Lemma. 

Thus, the strong maximum principle implies that~$P$ attains its maximum at a critical point of~$u$. As a consequence, we have that
\begin{equation*}
\|\nabla u \|_{L^{\infty}(\Omega)}^p \leq M\|f\|_{L^{\infty}([0,M])} +M\beta =  M\|f\|_{L^{\infty}([0,M])} + M + (N-1)\mathcal{M}_{0}^{-}\|\nabla u\|_{L^{\infty}(\Omega)}^{p-1}.
\end{equation*}
The desired result now follows easily.
\end{proof}
 
\section{Estimates on level sets}\label{Estimates on level sets}

In this section, we will make an in-depth study of the level sets of~$u$ using the deficits~$D_1(t)$ and~$D_2(t)$ in combination with the identity~\eqref{Identity}, under the assumptions of Theorem~\ref{Theorem 1.1}, which will allow us at a later stage to leverage the arguments presented in~\cite{CianchiEspositoFuscoTrombetti+2008+153+189} to prove Theorem~\ref{Theorem 1.1}. 

{F}rom now on, we will work under the assumptions of Theorem~\ref{Theorem 1.1}
and take~$\eps$ as in~\eqref{Definition of eps},
with~$D(\Omega)$ denoting the isoperimetric deficit of~$\Omega$
as given by~\eqref{e8239ryjkefg239yr3829}.

Using Lagrange's Mean Value Theorem, we see that 
$$
D_4\leq p'\left(\frac{\mathcal{H}^{N-1}(\partial \Omega)}{c_N |\Omega|^{1/N'}}\right)^{\frac{1}{p-1}}D(\Omega).
$$
Then, using~\eqref{Identity} and~\eqref{Gradient bound}, we see that there exists a positive constant~$\tilde{C}_1$, depending only on~$p$, $N$, $M$, $\|f\|_{L^{\infty}([0,M])}$, $|\Omega|$, $\mathcal{H}^{N-1}(\partial \Omega)$, and~$\mathcal{M}_{0}^{-}$, such that  
\begin{equation}\label{Estimate 4.1}
    \int_{0}^M W(t) D_i (t) dt \leq \tilde{C}_1\eps, \qquad {\mbox{ for }}\. i = 1,2 .
\end{equation}

The key to prove Theorem~\ref{Theorem 1.1} is to have a quantitative understanding of the level sets of the solution~$u$ of~\eqref{Problem}, by detecting how the level sets and super level sets of~$u$ relate to spheres and balls, respectively, and how~$|\nabla u|$ is close to be constant on~$\{u = t\}$. All the information needed for this investigation is contained in the deficits~$D_1$ and~$D_2$. 
To this end, the following lower bound for~$W(t)$ (defined in~\eqref{W}) will be useful.

\begin{lemma}\label{Positivity of W}
Let~$\Omega\subset \mathbb{R}^N$ be a bounded~$C^2$ domain and let~$f$ be a locally Lipschitz function such that~$f>\phi_0$ for some positive constant~$\phi_0$. 

Let~$u \in C^{1,\alpha}(\overline{\Omega})$ be a weak solution of~\eqref{Problem}. 

Assume that either~\ref{Assumption 1 on f} or~\ref{Assumption 2 on f} hold.
 
Then, there exists a constant~$C_1:= C_1(N,p,|\Omega|,\phi_0)>0$ such that, for all~$ t\in (0,M)$,
 $$
 W(t)\geq C_1.
 $$
\end{lemma}

\begin{proof}
In light of the definition of~$W$, we see that the cases~$p\in(1,N)$ and~$p\geq N$ are different. Because of this, we break the proof of the result into two parts, one dealing with the case~$p\in(1,N)$ and the other with the case~$p\geq N$.
    
    We start by proving the desired result when~$p\in(1,N)$.
    In this scenario, we have that
    \begin{equation*}
        \begin{split}
            \frac{I(t)}{\mu(t)}&= |\{u>t\}|^{-1}\int_{\{u>t\}} f(u(x))\,dx\leq  \frac{N}{N-p}|\{u>t\}|^{-1}\int_{\{u>t\}} \phi(u(x))\,dx\leq \frac{N}{N-p} \phi(t),
        \end{split}
    \end{equation*}
    where we have used the fact that~$\phi$ is nonincreasing.
    
    Consequently, if~\ref{Assumption 1 on f} holds, then
    \begin{equation*}
        \begin{split}
            W(t)&= p'f(t)\mu(t)^{\frac{p-N}{N(p-1)}} +\frac{p-N}{N(p-1)}I(t)\mu(t)^{\frac{p-pN}{N(p-1)}}\\
            &= \mu(t)^{\frac{p-N}{N(p-1)}}\left(p'f(t) + \frac{p-N}{N(p-1)}\frac{I(t)}{\mu(t)}\right)\\
            &\geq \mu(t)^{\frac{p-N}{N(p-1)}}\left(p'f(t) - \frac{1}{p-1}\phi(t)\right)\\
            &\geq \phi(t) \mu(t)^{\frac{p-N}{N(p-1)}}\\
            &\geq |\Omega|^{\frac{p-N}{N(p-1)}}\phi_0,
        \end{split}
    \end{equation*}as desired.
   
    When~$p\geq N$, the situation is much simpler because~$\frac{p-N}{N(p-1)}\geq 0$. Indeed,  when~$p = N$ we have that
    \begin{equation*}
            W(t)= p'f(t)\mu(t)^{\frac{p-N}{N(p-1)}} +\frac{p-N}{N(p-1)}I(t)\mu(t)^{\frac{p-pN}{N(p-1)}}= p' f(t).
    \end{equation*}
Instead, when~$p>N$,
    \begin{equation*}
        \begin{split}
            W(t)&= p'f(t)\mu(t)^{\frac{p-N}{N(p-1)}} +\frac{p-N}{N(p-1)}I(t)\mu(t)^{\frac{p-pN}{N(p-1)}}\\
            &\geq \frac{p-N}{N(p-1)}I(t)\mu(t)^{\frac{p-pN}{N(p-1)}}\\
            &= \frac{p-N}{N(p-1)}\mu(t)^{\frac{p-pN}{N(p-1)}}\int_{\{u>t\}}f(u) dx\\
            &\geq \frac{p-N}{N(p-1)}\mu(t)^{\frac{p-pN}{N(p-1)}+1}\phi(t)\\
            &\geq \frac{p-N}{N(p-1)}|\Omega|^{\frac{p-N}{N(p-1)}}\phi_0
        \end{split}
    \end{equation*}
and the proof of the desired result is thereby complete.
\end{proof}

Putting together~\eqref{Estimate 4.1} and Lemma~\ref{Positivity of W} we see that 
\begin{equation}\label{5.2}
    \int_{0}^{M} D_i (t)\,dt \leq C_1^{-1}\int_{0}^{M} W(t) D_i(t)\,dt\leq C_1^{-1} \tilde{C}_1\eps, \qquad{\mbox{ for }}\, i = 1,2.
\end{equation}

Let now
\begin{equation}\label{C 2}
C_2:= C_2(N,p,|\Omega|, \mathcal{H}^{N-1}(\partial\Omega),M,\|f\|_{L^{\infty}([0,M])}, \phi_0,\mathcal{M}_{0}^{-})=  C_1^{-1} \tilde{C}_1 >0
\end{equation}
where~$\tilde{C}_1$ is as in~\eqref{Estimate 4.1} and~$C_1$ is as in Lemma~\ref{Positivity of W}.
Using these estimates we can obtain results in the spirit of Section~2 in~\cite{CianchiEspositoFuscoTrombetti+2008+153+189}.

Throughout this section, we make the additional assumption that
\begin{equation}\label{eps <=1}
\eps \le 1
\end{equation}
and we will denote the isoperimetric deficit of the super-level set~$\{u>t\}$ by~$D(t)$, i.e.,
\[D(t) := \frac{\mathcal{H}^{N-1}(\{u = t\})-\mathcal{H}^{N-1}(\{u^* = t\})}{\mathcal{H}^{N-1}(\{u^* = t\})}.\]

Also, as in~\cite{CianchiEspositoFuscoTrombetti+2008+153+189}, we define~$t_{\alpha, \varepsilon}$ by 
\begin{equation}\label{def tae}
    t_{\alpha,\varepsilon} := \sup\left\{t >0 \; {\mbox{ s.t. }}\;|\{u>t\}|>\varepsilon^{\alpha N'}\right\},
\end{equation}
where~$\alpha >0$ will be determined later. 

\begin{lemma}\label{lemma "2.1"}
Let~$\Omega\subset \mathbb{R}^N$ be a bounded~$C^2$ domain and let~$f$ be a locally Lipschitz function such that~$f>\phi_0$ for some positive constant~$\phi_0$. 

Let~$u \in C^{1,\alpha}(\overline{\Omega})$ be a weak solution of~\eqref{Problem}. 

Assume that either~\ref{Assumption 1 on f} or~\ref{Assumption 2 on f} hold.

Then, there exists a set~$I_1 \subset (0,M)$ such that~$|I_1| \leq \varepsilon^{1/2}$
    and 
    \begin{equation}\label{Estimate 3.3}
        D(t) \leq \frac{C_2}{p'c_N^{p'}}\varepsilon^{\frac{1}{2}-p'\alpha}, \qquad {\mbox{ for all }} t \in (0,t_{\alpha,\varepsilon})\setminus I_1,
    \end{equation}
    or equivalently,
    \begin{equation}\label{Estimate 3.4}
        \mathcal{H}^{N-1}(\{u = t\})\leq \left(1+ \frac{C_2\varepsilon^{\frac{1}{2}-p'\alpha}}{p'c_N^{p'}}\right)\mathcal{H}^{N-1}(\{u^* = t\}), \qquad {\mbox{ for all }} t \in (0,t_{\alpha,\varepsilon})\setminus I_1,
    \end{equation}
where~$C_2$ is the constant defined by~\eqref{C 2}.    
\end{lemma}

\begin{proof}
Set
\begin{equation}\label{I1bedde}I_1 := \left\{
t\in (0,M)\;{\mbox{ s.t. }}\; D_2(t)>C_2\varepsilon^{1/2}\right\}.\end{equation} We observe that
$$\displaystyle\int_{0}^M D_2(t) dt \leq C_2\varepsilon, $$ thanks to~\eqref{5.2}, and consequently~$|I_1|\leq \varepsilon^{1/2}$.

Moreover, by the definition of~$I_1$, for every~$t \in (0,t_{\alpha,\varepsilon})\setminus I_1$, we have that~$D_2(t) \leq C_2\varepsilon^{1/2}$.
   Hence, employing Lagrange's Intermediate Value Theorem and the isoperimetric inequality we conclude that
    \begin{equation*}
    \begin{split}
      C_2 \varepsilon^{1/2}&\geq D_2(t) = \mathcal{H}^{N-1}(\{u=t\})^{p'}-\mathcal{H}^{N-1}(\{u^*=t\})^{p'}\\
       &\geq p' \mathcal{H}^{N-1}(\{u^*=t\})^{p'-1}\left(\mathcal{H}^{N-1}(\{u=t\})-\mathcal{H}^{N-1}(\{u^*=t\})\right)\\
       &= p'\mathcal{H}^{N-1}(\{u^*=t\})^{p'} D(t)\\
       &=  p' c_N^{p'}\mu(t)^{\frac{p'}{N'}} D(t), 
    \end{split} 
    \end{equation*}
    for any~$t\in (0,t_{\alpha,\eps})\setminus I_1$, where~$c_N$ is the sharp isoperimetric constant.
    
    Thus, by the definition of~$t_{\alpha,\eps}$ in~\eqref{def tae}, we obtain
    that, for all~$t \in (0,t_{\alpha,\eps})$,
    \begin{equation*}
        D(t)\leq \frac{C_2}{p'c_N^{p'}}\,\varepsilon^{\frac{1}{2}-p'\alpha},
    \end{equation*}as desired.
\end{proof}

\begin{lemma}\label{lemma "2.2"}
Let~$\Omega\subset \mathbb{R}^N$ be a bounded~$C^2$ domain and let~$f$ be a locally Lipschitz function such that~$f>\phi_0$ for some positive constant~$\phi_0$. 

Let~$u \in C^{1,\alpha}(\overline{\Omega})$ be a weak solution of~\eqref{Problem}. 

Assume that either~\ref{Assumption 1 on f} or~\ref{Assumption 2 on f} hold.
Let~$I_1$ be as in~\eqref{I1bedde}.

Then, for every~$t\in (0,t_{\alpha,\varepsilon})\setminus I_1$, there exists~$x_t \in \mathbb{R}^N$ such that 
    \begin{equation}\label{Estimate 3.7}
        |\{u>t\}\Delta B_{R_t}(x_t)|\leq C_3 R_{t}^N \varepsilon^{\frac{1}{4}-\frac{p'\alpha}{2}},
    \end{equation}
  where~$R_t := \left(\frac{|\{u>t\}|}{\omega_n}\right)^{1/N}$ and 
\begin{equation}\label{C 3}
C_3 :=C_3\left(N, p, |\Omega|, \mathcal{H}^{N-1}(\partial \Omega),\phi_0,\|f\|_{L^{\infty}([0,M])},\mathcal{M}_{0}^{-}\right)= \frac{\sqrt{C_2 \gamma(N)}}{\sqrt{p'\omega_{N}^{p'}}}\omega_N.
\end{equation}
Here, $\gamma(N)$ is the constant in the quantitative isoperimetric inequality~\eqref{Quantitative isoperimetric inequality}
and~$C_2$ is the constant in~\eqref{C 2}.
\end{lemma}
\begin{proof}
Applying the quantitative isoperimetric inequality~\eqref{Quantitative isoperimetric inequality} to the sets~$\{u>t\}$, with~$t\in (0,t_{\alpha,\varepsilon})\setminus I_1$, we see that there exists~$x_t\in\mathbb{R}^N$ such that 
    \begin{equation*}
        \frac{|\{u>t\}\Delta B_{R_t}(x_t)|^2}{|B_{R_t}|^2} = \min_{x\in \mathbb{R}^N} \left\{\frac{|\{u>t\}\Delta B_{R_t}(x)|^2}{|B_{R_t}|^2}\right\} \leq \gamma(N) D(t).
    \end{equation*}
    As a consequence, in light of~\eqref{Estimate 3.3} we have that 
    \begin{equation*}
        |\{u>t\}\Delta B_{R_t}(x_t)|\leq \frac{\sqrt{C_2 \gamma(N)}}{\sqrt{p'c_N^{p'}}}|B_{R_t}|\varepsilon^{\frac{1}{4}-\frac{p'\alpha}{2}} = \frac{\sqrt{C_2\gamma(N)}}{\sqrt{p'c_N^{p'}}}\omega_N R_{t}^N\varepsilon^{\frac{1}{4}-\frac{p'\alpha}{2}},
    \end{equation*}as desired.
\end{proof}

Now, for each~$t \in (0,M)$ we define 
\begin{equation}\label{def beta t}
    \beta_t := \frac{\left(\displaystyle\fint_{\{u=t\}}\frac{d\mathcal{H}^{N-1}}{|\nabla u|}\right)^{\frac{1}{p'}}}{\left(\displaystyle\fint_{\{u=t\}}|\nabla u|^{p-1} \,d\mathcal{H}^{N-1}\right)^{\frac{1}{p(p-1)}}}
\end{equation}
and we have the following result:

\begin{lemma}\label{lemma "2.4"}
Let~$\Omega\subset \mathbb{R}^N$ be a bounded~$C^2$ domain and let~$f$ be a locally Lipschitz function such that~$f>\phi_0$ for some positive constant~$\phi_0$. 

Let~$u \in C^{1,\alpha}(\overline{\Omega})$ be a weak solution of~\eqref{Problem}. 

Assume that either~\ref{Assumption 1 on f} or~\ref{Assumption 2 on f} hold.
Also, assume~\eqref{eps <=1}.

Then, there exist a constant
$$C_4:=C_4\left(N, p, |\Omega|, \mathcal{H}^{N-1}(\partial \Omega),\phi_0,\|f\|_{L^{\infty}([0,M])},\mathcal{M}_{0}^{-}\right)>0$$
and a set~$I_2 \subset (0,M)$ such that
   \begin{equation}\label{Estimate 3.10}
       |I_2| \leq \varepsilon^{1/2}
   \end{equation}
   and, for all~$t \in (0,t_{\varepsilon,\alpha})\setminus I_2$,
   \begin{equation}\label{Estimate 3.11}
     \fint_{\{u=t\}}\left|\frac{1}{|\nabla u|}- \beta_t\right|d\mathcal{H}^{N-1}\leq \begin{cases}
           \displaystyle C_4\beta_t \varepsilon^{\frac{p'}{2}-\alpha},
            &\text{ if } 1<p<2,\\
          \displaystyle  C_4\beta_t \varepsilon^{\frac{1}{2p}-\frac{\alpha}{p-1}},&\text{ if } p\ge 2.
        \end{cases} 
   \end{equation}
   
   Furthermore, for every~$t \in (0,t_{\varepsilon,\alpha})\setminus I_2$, there exists a Borel set~$U_t \subset \{u=t\}$ such that 
   \begin{equation}\label{Estimate 3.12}
      \left|\frac{1}{|\nabla u|}- \beta_t\right|\leq \begin{cases}
        \displaystyle   C_4\beta_t \varepsilon^{\frac{p'}{4}-\alpha}, & \text{ if }1<p<2,\\
\displaystyle   C_4\beta_t \varepsilon^{\frac{1}{4p}-\frac{\alpha}{p-1}},&\text{ if } p\ge2,
        \end{cases}
   \end{equation}
for~$\mathcal{H}^{N-1}-$a.e.~$x\in \{u = t\}\setminus U_t$,
   and
   \begin{equation}\label{Estimate 3.13}
       \mathcal{H}^{N-1}(U_t)\leq \begin{cases}
          \displaystyle \varepsilon^{\frac{p'}{4}}\mathcal{H}^{N-1}(\{u=t\}),& \text{ if }1<p<2\\
        \displaystyle   \varepsilon^{\frac{1}{4p}}\mathcal{H}^{N-1}(\{u=t\}),& \text{ if }p\ge2.
       \end{cases}
   \end{equation}
\end{lemma}

\begin{proof}
    We define 
\begin{equation*}
            I_2 := \big\{t\in(0,M)\;{\mbox{ s.t. }}\; D_1(t) >C_2\sqrt{\varepsilon}\big\}
        \end{equation*}
        and we observe that,
by~\eqref{5.2}, we have that~$|I_2|\leq \sqrt{\varepsilon}$. 

Furthermore, for every~$t\in (0,t_{\varepsilon,\alpha})\setminus I_2$, by the definitions of~$t_{\varepsilon,\alpha}$ and~$ I_2$ and the isoperimetric inequality, it holds that
    \begin{equation}\label{Estimate 3.14}
        \begin{split}
            C_2\sqrt{\varepsilon}&\geq D_1(t)= \mathcal{H}^{N-1}(\{u=t\})^{p'}\left(\left(\fint_{\{u=t\}}|\nabla u|^{p-1}d\mathcal{H}^{N-1}\right)^{\frac{1}{p-1}}\left(\fint_{\{u=t\}}|\nabla u|^{-1}d\mathcal{H}^{N-1}\right)-1\right)\\
            &\qquad\geq c_N^{p'} \varepsilon^{p'\alpha}\left(\left(\fint_{\{u=t\}}|\nabla u|^{p-1}d\mathcal{H}^{N-1}\right)^{\frac{1}{p-1}}\left(\fint_{\{u=t\}}|\nabla u|^{-1}d\mathcal{H}^{N-1}\right)-1\right).
        \end{split}
    \end{equation}
That is, 
    \begin{equation*}
        \left(\fint_{\{u=t\}}|\nabla u|^{p-1}d\mathcal{H}^{N-1}\right)\left(\fint_{\{u=t\}}|\nabla u|^{-1}d\mathcal{H}^{N-1}\right)-1 \leq C_2c_{N}^{-p'}\varepsilon^{\frac{1}{2}-p'\alpha}.
    \end{equation*}
    Hence, recalling~\eqref{Quantitative Holder},
    \begin{equation}\label{Estimate 3.15}
        \fint_{u^{-1}(t)}\left|\frac{1}{|\nabla u|}- \beta_t\right|d\mathcal{H}^{N-1}\leq \begin{cases}
            C_4\beta_t \left[\varepsilon^{\frac{p-1}{2}-p\alpha}+\varepsilon^{\frac{1}{2(p-1)}-\frac{p'\alpha}{p-1}}\right]^{1/p}, &\text{ if }1<p<2,\\
C_4\beta_t \varepsilon^{\frac{1}{2p}-\frac{\alpha}{p-1}},&\text{ if } p\ge2,
        \end{cases}  
    \end{equation}
    where~$C_4$ is a positive constant that depends on~$N$, $p$, $|\Omega|$, $\mathcal{H}^{N-1}(\partial \Omega)$,
    $\phi_0$, $\mathcal{M}_{0}^{-}$ and~$\|f\|_{L^{\infty}([0,M])}$ . 
    
    Since~$\varepsilon \le 1$ and~$\frac{p-2}{2}-p\alpha < \frac{1}{2(p-1)}-\frac{p'\alpha}{p-1}$ (for~$\alpha >0$ sufficiently small) we see that (possibly after adjusting the constant~$C_4$) that 
\begin{equation}\label{Estimate 3.15.1}
        \fint_{\{u=t\}}\left|\frac{1}{|\nabla u|}- \beta_t\right|d\mathcal{H}^{N-1}\leq \begin{cases}
          \displaystyle  C_4\beta_t \varepsilon^{\frac{p'}{2}-\alpha},&\text{ if }p\in(1,2),\\
         \displaystyle   C_4\beta_t \varepsilon^{\frac{1}{2p}-\frac{\alpha}{p-1}},& \text{ if }p\ge2,
        \end{cases}  
    \end{equation}

    To finish the proof, for every~$t \in (0,t_{\varepsilon,\alpha})\setminus I_2$ we define, when~$p\in(1,2)$,
    \begin{equation*}
        U_t := \displaystyle
            \left\{x \in \{u=t\}:\,\left|\frac{1}{|\nabla u|}-\beta_t\right|> C_4\beta_t \varepsilon^{\frac{p}{4}-\alpha}\right\}\end{equation*} and, when~$p\ge2$,
    \begin{equation*}         
U_t := \displaystyle
            \left\{x \in \{u=t\}:\,\left|\frac{1}{|\nabla u|}-\beta_t\right|> C_4\beta_t \varepsilon^{\frac{1}{4p}-\frac{\alpha}{p-1}}\right\}.
    \end{equation*}
    Now the desired result follows immediately from~\eqref{Estimate 3.15}.
\end{proof}

\begin{lemma}\label{lemma "2.6"}
Let~$\Omega\subset \mathbb{R}^N$ be a bounded~$C^2$ domain and let~$f$ be a locally Lipschitz function such that~$f>\phi_0$ for some positive constant~$\phi_0$. 

Let~$u \in C^{1,\alpha}(\overline{\Omega})$ be a weak solution of~\eqref{Problem}. 

Assume that either~\ref{Assumption 1 on f} or~\ref{Assumption 2 on f} hold.
Also, assume~\eqref{eps <=1} and let~$I:= I_1\cup I_2$. 

Then, there exists a positive constant
$$
C_5:=C_5\left(N, p, |\Omega|, \mathcal{H}^{N-1}(\partial \Omega),\phi_0,\|f\|_{L^{\infty}([0,M])},\mathcal{M}_{0}^{-}\right)
$$
such that,  for a.e.~$ t \in (0,t_{\alpha,\varepsilon})\setminus I$,
        \begin{equation*}
            \left|\frac{1}{|\nabla u^*|}_{\vert_{\{u*=t\}}}-\beta_t\right| \leq C_5\begin{cases}
             \displaystyle   \eps^{\frac{1}{2}-p'\alpha}\beta_t,& \text{ if }p\in(1,2),\\
           \displaystyle     \eps^{\frac{1}{2p}-\frac{\alpha}{p-1}}\beta_t,&\text{ if }p\ge2.
            \end{cases}
        \end{equation*}
\end{lemma}
\begin{proof}
    Applying the triangle inequality together with~\eqref{Equality of the derivative of dist fct} we find that
    \begin{equation}\label{Estimate 3.18}
    \begin{split}
        &\left|\frac{1}{|\nabla u^*|}- \beta_t\right|\\
        &= \left|\frac{1}{|\nabla u^*|}- \fint_{\{u=t\}} \frac{1}{|\nabla u|} d\mathcal{H}^{N-1}+ \fint_{\{u=t\}} \frac{1}{|\nabla u|} d\mathcal{H}^{N-1}- \beta_t\right|\\
        &\leq \left|\fint_{\{u^*=t\}}\frac{1}{|\nabla u^*|}d\mathcal{H}^{N-1}-\fint_{\{u=t\}} \frac{d\mathcal{H}^{N-1}}{|\nabla u |}\right| + \fint_{\{u=t\}}\left|\frac{1}{|\nabla u|}-\beta_t\right|d\mathcal{H}^{N-1}\\
        &= \left(\mathcal{H}^{N-1}(\{u^*=t\})^{-1}-\mathcal{H}^{N-1}(\{u=t\})^{-1}\right)\int_{\{u=t\}}\frac{d\mathcal{H}^{N-1}}{|\nabla u|}+ \fint_{\{u=t\}}\left|\frac{1}{|\nabla u|}-\beta_t\right|d\mathcal{H}^{N-1}\\
        &=D(t)\fint_{\{u=t\}}\frac{d\mathcal{H}^{N-1}}{|\nabla u|}+\fint_{\{u=t\}}\left|\frac{1}{|\nabla u|}-\beta_t\right|d\mathcal{H}^{N-1}\\
        &\leq D(t)\left(\beta_t+ \fint_{\{u=t\}}\left|\frac{1}{|\nabla u|}-\beta_t\right|d\mathcal{H}^{N-1}\right)+\fint_{\{u=t\}}\left|\frac{1}{|\nabla u|}-\beta_t\right|d\mathcal{H}^{N-1}\\
        &=D(t)\beta_t + \left(1+D(t)\right)\fint_{\{u=t\}}\left|\frac{1}{|\nabla u|}-\beta_t\right|d\mathcal{H}^{N-1}.
    \end{split}
    \end{equation}
    Using Lemmata~\ref{lemma "2.1"} and~\ref{lemma "2.4"} we see that there exists a positive constant~$\tilde{C}_4(N, p, |\Omega|, \mathcal{H}^{N-1}(\partial \Omega),\phi_0,\mathcal{M}_{0}^{-},\|f\|_{L^{\infty}([0,M])})$ such that, for all~$t \in (0,t_{\alpha,\varepsilon})\setminus I$,
        \begin{equation*}
            \left|\frac{1}{|\nabla u^*|}_{\Big\vert_{\{u*=t\}}}-\beta_t\right| \leq \tilde{C}_4\begin{cases}\displaystyle
                \left(\eps^{\frac{1}{2}-p'\alpha} +\eps^{\frac{p'}{2}-\alpha}+\eps^{\frac{1}{2}+\frac{p'}{2}-(p'+1)\alpha}\right)\beta_t, &\text{ if }p\in(1,2),\\
                \displaystyle
\left(\eps^{\frac{1}{2}-p'\alpha}+\eps^{\frac{1}{2}-p'\alpha+\frac{1}{2p}-\frac{\alpha}{p-1}}+\eps^{\frac{1}{2p}-\frac{\alpha}{p-1}}\right)\beta_t,& \text{ if }p\ge2.
            \end{cases}
        \end{equation*}
    Since we are working under the assumption~$\varepsilon \le 1$, the result follows with~$C_5 = 3 \tilde{C}_4$.
    \end{proof}

\section{Proof of Theorem~\ref{Quantitative estimate of Mu}}\label{Section 5}

To establish Theorem~\ref{Theorem 1.1} we need a good understanding of the size of the portion of the set~$\{|\nabla u|\leq \sigma\}$ (with~$\sigma \geq 0$) that lies in~$\Omega$. To this end, we define the following distribution-type function
\begin{equation}\label{dist fct of nu}
        M_u (\sigma):= |\{|\nabla u|\leq \sigma \}\cap \Omega|.
    \end{equation}

As stated in Theorem~\ref{Quantitative estimate of Mu}, we aim at finding a suitable estimate of~$M_u(\sigma)$ in terms of powers of~$\sigma$. 
Proving the estimate~\eqref{Estimate of Mu} for~$M_u$ requires two integral estimates originally contained in Theorems~2.2 and~2.3 in~\cite{MR2096703} and whose proofs we include here both to obtain more detailed statements and for the reader's convenience. 

We introduce two pieces of notation that will be useful throughout this section: given~$E\subset \Omega$, we define 
\begin{equation}\label{delta minus and delta plus}\begin{split}&
\delta_{-}(E):= \displaystyle\inf_{x\in E}\text{dist}(x,\partial\Omega)\\
\text{and }\quad &\delta_{+}(E):=\displaystyle\sup_{x\in E}\text{dist}(x,\partial\Omega).\end{split}
\end{equation}

\begin{lemma}[Theorem~2.2, \cite{MR2096703}]\label{First integral estimate} Let~$\Omega\subset \mathbb{R}^N$ be a bounded~$C^2$ domain and let~$f$ be a locally Lipschitz function such that~$f>\phi_0$ for some positive constant~$\phi_0$.

Let~$u \in C^{1,\alpha}(\overline{\Omega})$ be a weak solution of~\eqref{Problem}, with~$p\in(1,+\infty)$.
Let~$E\Subset\Omega$ be an open subset of~$\Omega$ and~$\delta_{-}(E)$ and~$\delta_{+}(E)$ be defined as in~\eqref{delta minus and delta plus}.

Then, there exists a positive constant~$C_{7}$, depending only on~$p$ and~$N$, such that, when~$2-p\leq \b<1$ and~$i = 1,\dots,N$, the following estimate holds true:
\begin{equation}\label{Second derivative estimate}
\begin{split}
    \int_{E}\frac{|\nabla u|^{p-2}}{|u_{x_i}|^{\beta}} |\nabla u_{x_i}|^2\,dx &\leq \frac{C_7}{(1-\beta)^2}\frac{\mathcal{V}(\delta_{+}(E))^{\frac{\beta}{p}}}{\delta_{-}(E)^2}\|\nabla u\|_{L^p(\Omega)}^{p-\beta} + \frac{C_7}{1-\beta}\int_{\Omega}|f'(u)||u_{x_i}|^{2-\beta}\,dx.
    \end{split}
\end{equation}
\end{lemma}

\begin{proof}
Let~$E$ be as stated above and, for simplicity's sake, we employ the notation, $\delta_{\pm} = \delta_{\pm}(E)$.
By Lemma~2.3 in~\cite{MR2096703} we know that, for every~$\phi \in H^1(\Omega)$ with compact support in~$\Omega$ and $i= 1,\dots, N$,
    \begin{equation}\label{Linearized equation}
        \int_{\Omega} \Big(
        |\nabla u|^{p-2}\langle \nabla u_{x_i}, \nabla \phi\rangle\,+\, (p-2)|\nabla u|^{p-4}\langle \nabla u\,,\,\nabla u_{x_i}\rangle \, \langle \nabla u\,,\,\nabla \phi\rangle \Big)\,dx = \int_{\Omega} f'(u)u_{x_i}\phi \,dx.
    \end{equation}
    To prove the result we construct an adequate test function to use in the above identity.
With that in mind, we let~$\varphi \in C^{\infty}_c(\Omega)$ be a nonnegative smooth compactly supported function such that 
\begin{equation}\label{the cutoff}
    \varphi \equiv 1\, \text{on} \, E,\qquad
    \text{supp}(\varphi) \subset E_{\delta_{-}/2}:= \big\{x\in \mathbb{R}^N\;{\mbox{ s.t. }}\; \text{dist}(E,x)<\delta_{-}/2\big\}, \qquad |\nabla \varphi| \leq \frac{C_{8}}{\delta_{-}},
\end{equation}
where~$C_8$ depends only on the dimension (we can take, for example, $C_{8} = 6 |B_1|$). 

By ~\eqref{Size of the tube} we have that
$$\left|\Omega\setminus E\right|\leq \left|\{x\in \Omega\;{\mbox{ s.t. }}\; \text{dist}(x, \partial \Omega)\leq \delta_+\}\right|\leq \mathcal{M}(\delta_+).$$
Also, we define 
\begin{equation*}
    G_{\alpha}(s):= \begin{cases}
        0,& \, \text{if}\, |s|\leq \alpha,\\
        2s-2\alpha,&\,\text{if}\, \alpha\leq s\leq 2\alpha,\\
        2\alpha+2s,&\,\text{if}\, -2\alpha\leq s\leq -\alpha,\\
        s,&\, \text{if}\, |s|\geq 2\alpha.
    \end{cases}
\end{equation*}
Note that~$G_{\alpha}$ is Lipschitz continuous, that~$|G'_{\alpha}(s)|\leq 2$ for $\mathcal{L}^1-$a.e.~$s\in \mathbb{R}$, and that~$G_{\alpha}(s) \to s$ pointwise as~$\alpha \to 0^+$.

For~$\beta<1$, we define 
\begin{equation*}
    \psi_{\alpha} (x) := \frac{G_{\alpha}(u_{x_i}(x))}{|u_{x_i}(x)|^{\beta}}\left(\varphi(x)\right)^2
\end{equation*}
and we remark that, due to the properties of~$G_{\alpha}$, the function~$\psi_{\alpha}$ is an admissible test function for~\eqref{Linearized equation}. 

Hence, using~$\psi_{\alpha}$ into~\eqref{Linearized equation} we obtain that
\begin{equation*}
    \begin{split}&
        \int_{\Omega} f'(u)u_{x_i}\psi_{\alpha} \,dx\,\\&\quad=\, \int_{\Omega} \left[|\nabla u|^{p-2}\langle \nabla u_{x_i}, \nabla \psi_{\alpha}\rangle\,+\, (p-2)|\nabla u|^{p-4}\langle \nabla u\,,\,\nabla u_{x_i}\rangle \, \langle \nabla u\,,\,\nabla \psi_{\alpha}\rangle \right]\,dx \\
        &\quad=\,\int_{\Omega} \left[|\nabla u|^{p-2}\varphi^2\left\langle \nabla u_{x_i}, \nabla\left(\frac{G_{\alpha}(u_{x_i})}{|u_{x_i}|^{\beta}}\right) \right\rangle\right]\,dx\\
        &\qquad+(p-2)\int_{\Omega} \left[|\nabla u|^{p-4}\varphi^2\langle \nabla u,\nabla u_{x_i}\rangle \left\langle \nabla u,\nabla \left(\frac{G_{\alpha}(u_{x_i})}{|u_{x_i}|^{\beta}}\right)\right\rangle\right]\,dx\\
        &\qquad+\, 2\int_{\Omega} \left[\varphi\frac{|\nabla u|^{p-2}}{|u_{x_i}|^{\beta}}G_{\alpha}(u_{x_i})\langle \nabla u_{x_i}, \nabla \varphi \rangle\right]\,dx\\
        &\qquad+ 2(p-2)\int_{\Omega}\left[\varphi\frac{|\nabla u|^{p-4}}{|u_{x_i}|^{\beta}}G_{\alpha}(u_{x_i})\langle \nabla u\,,\,\nabla u_{x_i}\rangle \langle \nabla u\,,\,\nabla \varphi\rangle \right]\,dx.
    \end{split}
\end{equation*}
This, after noticing that~$\nabla \left(\frac{G_{\alpha}(u_{x_i})}{|u_{x_i}|^{\beta}}\right) = |u_{x_i}|^{-\beta}\left(G_{\alpha}'(u_{x_i})-\beta\frac{G_{\alpha}(u_{x_i})}{u_{x_i}}\right)\nabla u_{x_i}$, becomes
\begin{equation*}
    \begin{split}
        \int_{\Omega} f'(u)u_{x_i}\psi_{\alpha} \,dx\,&=\, \int_{\Omega} \left[\varphi^2\frac{|\nabla u|^{p-2}}{|u_{x_i}|^{\beta}} |\nabla u_{x_i}|^2 \left(G_{\alpha}'(u_{x_i})-\beta\frac{G_{\alpha}(u_{x_i})}{u_{x_i}}\right)\right]\,dx\\
        &\qquad+\,(p-2)\int_{\Omega} \left[\varphi^2|\nabla u|^{p-4}\varphi\left|\langle \nabla u,\nabla u_{x_i}\rangle\right|^2 \left(G_{\alpha}'(u_{x_i})-\beta\frac{G_{\alpha}(u_{x_i})}{u_{x_i}}\right) \right]\,dx\\
        &\qquad+\, 2\int_{\Omega} \left[\varphi\frac{|\nabla u|^{p-2}}{|u_{x_i}|^{\beta}}G_{\alpha}(u_{x_i})\langle \nabla u_{x_i}, \nabla \varphi \rangle\right]\,dx\\
        &\qquad+ 2(p-2)\int_{\Omega}\left[\varphi\frac{|\nabla u|^{p-4}}{|u_{x_i}|^{\beta}}G_{\alpha}(u_{x_i})\langle \nabla u,\nabla u_{x_i}\rangle \langle \nabla u,\nabla \varphi\rangle \right]\,dx.
    \end{split}
\end{equation*}

Now we denote by~$T$ the first term on the right-hand side of the equation above.
Furthermore, the facts that~$\beta< 1$ and~$G_{\alpha}'(s)\leq 2$ imply that~$G_{\alpha}'(u_{x_i})-\beta\frac{G_{\alpha}(u_{x_i})}{u_{x_i}} \geq 0$.  As a result, using the Cauchy-Schwartz inequality, when~$p\leq 2$ we have that
\begin{equation*}
    \begin{split}
        T &= \int_{\Omega} f'(u)u_{x_i}\psi_{\alpha} \,dx\\
        &\qquad+ (2-p)\int_{\Omega} \left[|\nabla u|^{p-4}\varphi^2\left|\langle \nabla u,\nabla u_{x_i}\rangle\right|^2 \left(G_{\alpha}'(u_{x_i})-\beta\frac{G_{\alpha}(u_{x_i})}{u_{x_i}}\right) \right]dx\\
        &\qquad- 2\int_{\Omega} \left[\varphi\frac{|\nabla u|^{p-2}}{|u_{x_i}|^{\beta}}G_{\alpha}(u_{x_i})\langle \nabla u_{x_i}, \nabla \varphi \rangle \right]dx\\
        &\qquad+ 2(2-p)\int_{\Omega}\left[\varphi\frac{|\nabla u|^{p-4}}{|u_{x_i}|^{\beta}}G_{\alpha}(u_{x_i})\langle \nabla u,\nabla u_{x_i}\rangle \langle \nabla u,\nabla \varphi\rangle \right]dx\\
        &\leq 2\int_{\Omega}|f'(u)||u_{x_i}|^{2-\beta}\varphi^2\,dx\,+\,2(3-p)\int_{\Omega}\varphi\frac{|\nabla u|^{p-2}}{|u_{x_i}|^{\beta}}|\nabla u_{x_i}||\nabla \varphi| \left|G_{\alpha}(u_{x_i})\right|\,dx\\
        &\qquad+ (2-p)\int_{\Omega} \left[\varphi^2\frac{|\nabla u|^{p-2}}{|u_{x_i}|^{\beta}} |\nabla u_{x_i}|^2 \left(G_{\alpha}'(u_{x_i})-\beta\frac{G_{\alpha}(u_{x_i})}{u_{x_i}}\right)\right]dx\\
        &\leq  2\int_{\Omega} \varphi^2|f'(u)||u_{x_i}|^{2-\beta}\,dx\, +\, (2-p)T + 2(3-p)\int_{\Omega}\varphi\frac{|\nabla u|^{p-2}}{|u_{x_i}|^{\beta}}|\nabla u_{x_i}||\nabla \varphi| \left|G_{\alpha}(u_{x_i})\right|\,dx,
    \end{split}
\end{equation*}
and, when~$p>2$,
\begin{equation*}
    \begin{split}
        T&\leq  2\int_{\Omega} \varphi^2|f'(u)||u_{x_i}|^{2-\beta}\,dx + 2(p-1) \int_{\Omega}\varphi\frac{|\nabla u|^{p-2}}{|u_{x_i}|^{\beta}}|\nabla u_{x_i}||\nabla \varphi| \left|G_{\alpha}(u_{x_i})\right|\,dx.
    \end{split}
\end{equation*}

Thus, there exists a positive constant~$C_p$ such that
\begin{equation*}
    \begin{split}
        &\int_{\Omega}\left[\varphi^2\frac{|\nabla u|^{p-2}}{|u_{x_i}|^{\beta}} |\nabla u_{x_i}|^2 \left(G_{\alpha}'(u_{x_i})-\beta\frac{G_{\alpha}(u_{x_i})}{u_{x_i}}\right)\right]\,dx\\
        &\leq C_p\int_{\Omega}\frac{|\nabla u|^{p-2}}{|u_{x_i}|^{\beta}}|\nabla u_{x_i}||\nabla \varphi| \left|G_{\alpha}(u_{x_i})\right|\,dx
+ 2C_p\int_{\Omega}|f'(u)||u_{x_i}|^{2-\beta}\varphi^2\,dx.
    \end{split}
\end{equation*}
Applying the Young's inequality we obtain that
\begin{equation*}
    \begin{split}&
        \int_{\Omega} \varphi\frac{|\nabla u|^{p-2}}{|u_{x_i}|^{\beta}}|\nabla u_{x_i}||\nabla \varphi| \left|G_{\alpha}(u_{x_i})\right|\,dx\\&\quad\leq \frac{1-\beta}{4C_p}\int_{\Omega}\varphi^2 \frac{|\nabla u|^{p-2}}{|u_{x_i}|^{\beta}}|\nabla u_{x_i}|^2  dx+ \frac{C_p}{1-\beta}\int_{\Omega}\frac{|\nabla u|^{p-2}}{|u_{x_i}|^{\beta}}\left|G(u_{x_i})\right|^{2}|\nabla \varphi|^2dx\\
        &\quad\leq \frac{1-\beta}{4C_p}\int_{\Omega}\varphi^2 \frac{|\nabla u|^{p-2}}{|u_{x_i}|^{\beta}}|\nabla u_{x_i}|^2 dx
        +\frac{4C_p}{1-\beta}\int_{\Omega}|\nabla \varphi|^2|\nabla u|^{p-\beta}dx.
    \end{split}
\end{equation*}
Exploiting also the H\"older's inequality together with the properties of~$\varphi$, we see that 
\begin{equation*}
    \begin{split}
        \int_{\Omega} \varphi\frac{|\nabla u|^{p-2}}{|u_{x_i}|^{\beta}}|\nabla u_{x_i}||\nabla \varphi| \left|G_{\alpha}(u_{x_i})\right|dx&\leq \frac{1-\beta}{4C_p}\int_{\Omega}\varphi^2 \frac{|\nabla u|^{p-2}}{|u_{x_i}|^{\beta}}|\nabla u_{x_i}|^2 dx\,+\frac{4C_{8}^2 C_p}{1-\beta}\frac{\mathcal{V}(\delta_{+})^{\frac{\beta}{p}}}{\delta_{-}^2}\|\nabla u\|_{L^p(\Omega)}^{p-\beta}.
    \end{split}
\end{equation*}

Therefore,
\begin{equation*}
    \begin{split}
      &  \int_{\Omega}\left[\varphi^2\frac{|\nabla u|^{p-2}}{|u_{x_i}|^{\beta}} |\nabla u_{x_i}|^2 \left(G_{\alpha}'(u_{x_i})-\beta\frac{G_{\alpha}(u_{x_i})}{u_{x_i}}\right)\right]dx\\
      &\leq  \frac{1-\beta}{4}\int_{\Omega}\varphi^2 \frac{|\nabla u|^{p-2}}{|u_{x_i}|^{\beta}}|\nabla u_{x_i}|^2 dx
      +\frac{\tilde{C}_7}{1-\beta}\frac{\mathcal{V}(\delta_{+})^{\frac{\beta}{p}}}{\delta_{-}^2}\|\nabla u\|_{L^p(\Omega)}^{p-\beta}
+ \tilde{C}_7\int_{\Omega}|f'(u)||u_{x_i}|^{2-\beta}\varphi^2\,dx,
    \end{split}
\end{equation*}
where~$\tilde{C}_{7} := \max\{4C_{8}^2 C_{p}^2, 2C_p\}$. 

Taking the limit as~$\alpha \to 0^+$, we thus obtain that
\begin{equation*}
    \begin{split}
        \frac{3(1-\beta)}{4}\int_{\Omega} \varphi^2\frac{|\nabla u|^{p-2}}{|u_{x_i}|^{\beta}} |\nabla u_{x_i}|^2\,dx &\leq \frac{C_7}{1-\beta}\frac{\mathcal{V}(\delta_{+})^{\frac{\beta}{p}}}{\delta_{-}^2}\|\nabla u\|_{L^p(\Omega)}^{p-\beta} + C_7\int_{\Omega}|f'(u)||u_{x_i}|^{2-\beta}\,dx.
    \end{split}
\end{equation*}
Hence,
\begin{equation*}
    \begin{split}
        \int_{E}\frac{|\nabla u|^{p-2}}{|u_{x_i}|^{\beta}} |\nabla u_{x_i}|^2\,dx &\leq \frac{C_7}{(1-\beta)^2}\frac{\mathcal{V}(\delta_{+})^{\frac{\beta}{p}}}{\delta_{-}^2}\|\nabla u\|_{L^p(\Omega)}^{p-\beta} + \frac{C_7}{1-\beta}\int_{\Omega}|f'(u)||u_{x_i}|^{2-\beta}\,dx,
    \end{split}
\end{equation*}
where~$C_7 := 2\tilde{C}_7$.
\end{proof}

\begin{lemma}[Theorem 2.3,~\cite{MR2096703}]\label{Second integral estimate}
Let~$\Omega\subset \mathbb{R}^N$ be a bounded~$C^2$ domain and let~$f$ be a locally Lipschitz function such that~$f>\phi_0$ for some positive constant~$\phi_0$.

Let~$u \in C^{1,\alpha}(\overline{\Omega})$ be a weak solution of~\eqref{Problem}, with~$p\in(1,+\infty)$.

Let~$E\Subset \Omega$ be an open subset of~$\Omega$ and~$\d_+(E)$ and~$\d_{-}(E)$ be defined as in~\eqref{delta minus and delta plus}.

Then, there exists a positive constant~$C_{9}$, depending only on~$N$, $p$,
and~$\phi_0$, such that, when
$$\max\left\{0,\frac{p-2}{p-1}\right\}<r<1$$ the following estimate holds true:
    \begin{equation}\label{Summability of the gradient}
    \begin{split}
   &  \int_E \frac{dx}{|\nabla u(x)|^{(p-1)r}}\\
     &\leq C_{9}\frac{\mathcal{V}(\delta_{+}(E))^{\frac{1-r+rp}{p}}}{\delta_{-}(E)}\|\nabla u\|_{L^p(\Omega)}  +\frac{C_{9}}{(1-r)^2}\frac{\mathcal{V}(\delta_{+}(E))^{\frac{(p-1)r -p+2}{p}}}{\delta_{-}(E)^2}\|\nabla u\|_{L^p(\Omega)}^{(p-1)(2-r)}\\
        &\qquad\qquad+ \frac{C_{9}}{1-r}\|f'\|_{L^{\infty}([0,M])}\|\nabla u\|_{L^p(\Omega)}^{p-(p-1)r}.
        \end{split}
    \end{equation}
\end{lemma}

\begin{proof}
    The idea to prove~\eqref{Summability of the gradient} is to use an appropriate test function in the weak formulation of the equation for~$u$. As previously done, we use the short notation~$\delta_{\pm}= \delta_{\pm}(E)$.
    Let~$\varphi$ be a nonnegative smooth compactly supported (in~$\Omega$) cutoff function such that 
    \begin{equation*}
        \varphi \geq 0,\qquad \varphi \equiv 1,\,\text{on}\, E, \qquad
         \text{supp}\varphi\subset E_{\delta_{-}/2},\qquad |\nabla \varphi|\leq \frac{C_7}{\delta_{-}},
    \end{equation*}
    where~$C_8$ is as in~\eqref{the cutoff}.
    
    Given~$\alpha>0$, we let
    \begin{equation*}
        \psi_{\alpha} := \frac{\varphi}{\left(|\nabla u|^{p-1}+\alpha\right)^{r}}.
    \end{equation*}
    Thus, we have that
    \begin{equation*}
        \begin{split}&\!\!\!\!\!\!
           \phi_0 \int_{E_{\delta_{-}/2}} \frac{\varphi}{\left(|\nabla u|^{p-1} + \alpha\right)^r}\,dx \\&\leq \int_{\Omega} \psi_{\alpha} f(u)\,dx\\
           &= \int_{\Omega} |\nabla u|^{p-2}\langle \nabla u, \nabla \psi_{\alpha}\rangle\, dx \\
           &= \int_{\Omega} \frac{|\nabla u|^{p-2}}{\left(|\nabla u |^{p-1} +\alpha\right)^r}\langle \nabla u, \nabla \varphi\rangle\,dx\\
           &\qquad+ \int_{\Omega}\varphi |\nabla u|^{p-2} \left\langle \nabla u, \nabla \left(\frac{1}{\left(|\nabla u|^{p-1}+\alpha\right)^r)}\right) \right\rangle \,dx\\
           &\leq \int_{\Omega}|\nabla u|^{(p-1)(1-r)}|\nabla \varphi|dx + r\int_{\Omega}\varphi \frac{|\nabla u|^{p-2}}{\left(|\nabla u|^{p-1}+\alpha\right)^{r+1}} \langle \nabla u, \nabla (|\nabla u|^{p-1})\rangle\,dx\\
           &\leq \frac{C_{8}}{\delta_{-}}\int_{E_{\delta_{-}/2}\setminus E}|\nabla u|^{(p-1)(1-r)}\, dx+ r\int_{\Omega}\varphi \frac{|\nabla (|\nabla u|^{p-1})|}{\left(|\nabla u|^{p-1}+\alpha\right)^{r}}\,dx\\
           &\leq  \frac{C_{8}|E_{\delta_{-}/2}\setminus E|^{\frac{1-r+pr}{p}}}{\delta_{-}}\|\nabla u\|_{L^p(\Omega)}+ r\int_{\Omega}\varphi \frac{|\nabla (|\nabla u|^{p-1})|}{\left(|\nabla u|^{p-1}+\alpha\right)^{r}}\,dx \\
           &\leq C_{8}\frac{\mathcal{V}(\delta_+)^{\frac{1-r+rp}{p}}}{\delta_{-}}\|\nabla u\|_{L^p(\Omega)} + r\int_{E_{\delta_{-}/2}}\varphi\frac{|\nabla (|\nabla u|^{p-1})|}{\left(|\nabla u|^{p-1}+\alpha\right)^{r}}\,dx.
        \end{split}
    \end{equation*}
    To estimate the last term above, we apply Lemma~\ref{First integral estimate}. For this, we define
    $$
    \beta := (p-1)r-p+2<1.
    $$
    Then, by Young's inequality, 
    \begin{equation*}
        \begin{split}
            \int_{E_{\delta_{-}/2}}\varphi\frac{|\nabla (|\nabla u|^{p-1})|}{\left(|\nabla u|^{p-1}+\alpha\right)^{r}}\,dx&\leq (p-1)\int_{E_{\delta_{-}/2}}\varphi\frac{|\nabla u|^{\frac{p-2+ \beta}{2}+ \frac{p-2-\beta}{2}}|\nabla^2 u|}{|\nabla u|^{\frac{(p-1)r}{2}}\left(|\nabla u|^{p-1}+\alpha\right)^{r/2}}\,dx\\
            &= (p-1)\int_{E_{\delta_{-}/2}}\varphi\frac{|\nabla u|^{\frac{p-2-\beta}{2}}|\nabla^2 u|}{\left(|\nabla u|^{p-1}+\alpha\right)^{r/2}}\,dx\\
            &\leq \frac{\phi_0}{4}\int_{E_{\delta_{-}/2}}\frac{\varphi}{\left(|\nabla u|^{p-1}+\alpha\right)^r}\,dx + \frac{(p-1)^2}{\phi_0}\int_{E_{\delta_{-}/2}}\varphi|\nabla u|^{p-2-\beta}|\nabla^2 u|^2 \,dx.
        \end{split}
    \end{equation*}
    Note that, by construction,
 $$
 \delta_{-}\left(E_{\delta_{-}/2}\right) \geq \delta_{-}/2\qquad 
{\mbox{and}} 
\qquad \delta_+\left(E_{\delta_{-}/2}\right)\leq \delta_+.
 $$
Noting that the assumption~$r> \frac{p-2}{p-1}$ ensures that~$\beta>0$, and using the estimate in Lemma~\ref{First integral estimate}, we find that
    \begin{equation*}
    \begin{split}
         \int_{E_{\delta_{-}/2}}|\nabla u|^{p-2-\beta}|\nabla^2 u|^2 \,dx&\leq \frac{C_{10}}{(1-\beta)^2}\frac{\mathcal{V}(\delta_{+})^{\frac{\beta}{p}}}{\delta_{-}^2}\|\nabla u\|_{L^p(\Omega)}^{p-\beta} + \frac{C_{10}}{1-\beta}\int_{\Omega}|f'(u)||u_{x_i}|^{2-\beta}\,dx,
    \end{split}
    \end{equation*}
    for some positive constant~$C_{10}$ depending only on~$p$ and~$N$.
    
Accordingly, taking into account that~$\varphi\leq 1$, we see that
    \begin{equation*}
    \begin{split}
     &   \frac{3\phi_0}{4}\int_{E_{\delta_{-}/2}}\frac{\varphi}{\left(|\nabla u|^{p-1}+\a\right)^r}\,dx\\
     &\leq  C_{8}\frac{\mathcal{V}(\delta_+)^{\frac{1-r+rp}{p}}}{\delta_{-}}\|\nabla u\|_{L^p(\Omega)} + \frac{(p-1)^2C_{10}}{\phi_0 (1-\beta)^2}\frac{\mathcal{V}(\delta_{+})^{\beta/p}}{\delta_{-}^2}\|\nabla u\|_{L^p(\Omega)}^{p-\beta}
     + \frac{C_{10}(p-1)^2}{(1-\beta)\phi_0}\int_{\Omega}|f'(u)||\nabla u|^{2-\beta}\,dx.
    \end{split}
    \end{equation*}
    Taking the limit as~$\alpha \to 0^+$ and using again the fact that~$\varphi\leq 1$, we thereby conclude that 
    \begin{equation*}
        \begin{split}
       &     \int_E \frac{1}{|\nabla u|^{(p-1)r}}\,dx\\
       &\leq   C_{11}\frac{\mathcal{V}(\delta_+)^{\frac{1-r+rp}{p}}}{\delta_{-}}\|\nabla u\|_{L^p(\Omega)} + \frac{C_{11}}{ (1-\beta)^2}\frac{\mathcal{V}(\delta_{+})^{\beta/p}}{\delta_{-}^2}\|\nabla u\|_{L^p(\Omega)}^{p-\beta}
       + \frac{C_{11}}{(1-\beta)}\int_{\Omega}|f'(u)||\nabla u|^{2-\beta}\,dx,
        \end{split}
    \end{equation*}
    for some positive constant~$C_{11}$ depending only on~$p$, $N$ and~$\phi_0$. 
    
    In addition, by the definition of~$\beta$, we have that~$p-\beta = (p-1)(2-r)$, that~$p+\beta -2= (p-1)r$, and that~$2-\beta = p -(p-1)r$. 
    
    Therefore,
    \begin{equation*}
        \begin{split}
            \int_E \frac{1}{|\nabla u|^{(p-1)r}}\,dx&\leq C_{11}\frac{\mathcal{V}(\delta_+)^{\frac{1-r+rp}{p}}}{\delta_{-}}\|\nabla u\|_{L^p(\Omega)}   +\frac{C_{11}}{ (p-1)^2(1-r)^2}\frac{\mathcal{V}(\delta_{+})^{\frac{(p-1)r -p+2}{p}}}{\delta_{-}^2}\|\nabla u\|_{L^p(\Omega)}^{(p-1)(2-r)}\\
            &\qquad+ \frac{C_{11}}{(p-1)(1-r)}\int_{\Omega}|f'(u)||\nabla u|^{1+(p-1)(1-r)}\,dx\\
            &\leq C_{11}\frac{\mathcal{V}(\delta_+)^{\frac{1-r+rp}{p}}}{\delta_{-}}\|\nabla u\|_{L^p(\Omega)}   +\frac{C_{11}}{ (p-1)^2(1-r)^2}\frac{\mathcal{V}(\delta_{+})^{\frac{(p-1)r -p+2}{p}}}{\delta_{-}^2}\|\nabla u\|_{L^p(\Omega)}^{(p-1)(2-r)}\\
            &\qquad+ \frac{C_{11}}{(p-1)(1-r)}|\Omega|^{\frac{p-1}{p}r}\|f'\|_{L^{\infty}([0,M])}\|\nabla u\|_{L^p(\Omega)}^{p-(p-1)r},
        \end{split}
    \end{equation*}
    and the result follows with~$C_{9} := \max\left\{\frac{1}{(p-1)^2},1\right\} \, C_{11}$.
\end{proof}

\begin{remark} {\rm
    We point out that Lemmata~\ref{First integral estimate} and~\ref{Second integral estimate} could be expanded to encompass the cases of~$\beta<1$ and~$0<r<1$, but  we will not need these results in this additional generality
    for the purpose of this paper.}
\end{remark}

Now we have the tools to prove Theorem~\ref{Quantitative estimate of Mu}.

\begin{proof}[Proof of Theorem~\ref{Quantitative estimate of Mu}]
Let $$r: = \frac{p-2+|p-2|+2(p-1)}{4(p-1)}\in \left(\max\left\{\frac{p-2}{p-1},0\right\},1\right)$$
and
$$\theta :=(p-1)r = \frac{p-2+|p-2|+2p-2}{4}.$$ 
We start by assuming that~$\sigma \leq 1$. In this setting, we define
$$E:= \left\{x \in \Omega\;{\mbox{ s.t. }} \text{dist}(x,\partial \Omega)> \sigma^{\frac{\theta}{3}}\right\}.$$
Then, by construction, we see that~$\delta_{\pm}(E) = \sigma^{\frac{\theta}{3}}$.
Also, by Lemma~\ref{Second integral estimate},
\begin{equation*}
    \begin{split}
        &\left|\{x \in E: |\nabla u|\leq \sigma\}\right|\\
        \leq\,& \int_{E}\frac{\sigma^{(p-1)r}}{|\nabla u|^{(p-1)r}}\,dx\\
        \leq\,& C_{9}\frac{\mathcal{V}(\delta_+(E)^{\frac{1-r+rp}{p}}}{\delta_{-}(E)}\|\nabla u\|_{L^p(\Omega)}\sigma^{(p-1)r}
        +\frac{C_{9}}{(1-r)^2}\frac{\mathcal{V}(\delta_{+}(E))^{\frac{(p-1)r -p+2}{p}}}{\delta_{-}(E)^2}\|\nabla u\|_{L^p(\Omega)}^{(p-1)(2-r)} \sigma^{(p-1)r}\\
        &\qquad+ \frac{C_{9}}{1-r}\|f'\|_{L^{\infty}([0,M])}\|\nabla u\|_{L^p(\Omega)}^{p-(p-1)r} \sigma^{(p-1)r}\\
	\leq\,& \tilde{C}_{6}\sigma^{\frac{\theta}{3}},
    \end{split}
\end{equation*}
where~$\tilde{C}_{6}$ is a positive constant depending on~$N$, $p$, $|\Omega|$, $\mathcal{H}^{N-1}(\partial \Omega)$, $\phi_0$, $M$, and~$\|f'\|_{L^{\infty}([0,M])}$.
Note that to obtain the last inequality, we used the facts that~$\sigma\leq 1$,
$\mathcal{V}(\delta_{+}(E))\leq |\Omega|$ and
$$\|\nabla u\|_{L^p(\Omega)}\leq |\Omega|^{\frac{1}{p}}\|f\|_{L^{\infty}([0,M])}^{\frac{1}{p}}M^{\frac{1}{p}}.$$

Furthermore, by~\eqref{Size of the tube},
\begin{equation*}
    \left|\{x \in \Omega\setminus E\;{\mbox{ s.t. }}\; |\nabla u|\leq \sigma\}\right|\leq |\Omega\setminus E| = \mathcal{V}(\delta_{+}(E))\leq \left(1+\mathcal{M}_{0}^{-}\sigma^{\frac{\theta}{3}}\right)\mathcal{H}^{N-1}(\partial \Omega)\sigma^{\frac{\theta}{3}}\leq \left(1+\mathcal{M}_{0}^{-}\right)\sigma^{\frac{\theta}{3}}.
\end{equation*}
Hence,  there exists a positive constant~$C_{6}$, depending only on~$N$, $p$, $|\Omega|$, $\mathcal{H}^{N-1}(\partial \Omega)$, $\phi_0$, $M$, $\|f'\|_{L^{\infty}([0,M])}$, and~$\mathcal{M}_{0}^{-}$, such that
$$
M_{u}(\sigma)\leq C_{6}\sigma^{\frac{\theta}{3}}.
$$
This proves the desired result whenever~$\sigma \leq 1$. 

If instead~$\sigma \geq 1$, we clearly have that 
$$
M_u(\sigma)\leq |\Omega|\leq |\Omega|\sigma^{\frac{\theta}{3}}.
$$
The desired result then follows in this case by choosing~$C' := \max\{C_6, |\Omega|\}$.
\end{proof}

\section{\textbf{Proof of Theorem~\ref{Theorem 1.1}}}\label{sec7}
This section is devoted to the proof of Theorem~\ref{Theorem 1.1}.
Our strategy relies on applying some arguments present in Section~4 of~\cite{CianchiEspositoFuscoTrombetti+2008+153+189} in combination with the estimates obtained here in Section~\ref{Section 5}.

To this end, we present an upper bound on the size of the preimage of the set~$I$.

\begin{lemma}\label{Measure of preimage of I}
Let~$\Omega\subset \mathbb{R}^N$ be a bounded~$C^2$ domain and let~$f$ be a locally Lipschitz function such that~$f>\phi_0$ for some positive constant~$\phi_0$. 

Let~$u \in C^{1,\alpha}(\overline{\Omega})$ be a weak solution of~\eqref{Problem}. 

Assume that either~\ref{Assumption 1 on f} or~\ref{Assumption 2 on f} hold.

Then, for any~$\sigma >0$,
    \begin{equation*}
        \left|u^{-1}(I)\right|\leq M_u(\sigma)+ 2|\Omega|\|f\|_{L^{\infty}([0,M])}\frac{\eps^{1/2}}{\sigma^p},
    \end{equation*}
    where~$M_u$ is defined as in~\eqref{dist fct of nu}.
\end{lemma}

\begin{proof}
    By the coarea formula,  
    \begin{equation*}
    \begin{split}
        \left|u^{-1}(I)\right|&= \left|u^{-1}(I)\cap \{|\nabla u|\leq \sigma\}\right|+ \left|u^{-1}(I)\cap \{|\nabla u|>\sigma\}\right|\\
		&\leq M_{u}(\sigma)+ \int_{I}\int_{\{u = t\}\cap \{|\nabla u| > \sigma\}}\frac{d\mathcal{H}^{N-1}}{|\nabla u|}\,dt\\
              &\leq M_{u}(\sigma) + \int_I \sigma^{-1}\mathcal{H}^{N-1}\left(\{u=t\} \cap \{|\nabla u|>\sigma\}\right)\,dt.
    \end{split}
    \end{equation*}
    To estimate the last term, we use~\eqref{I}, finding that
\begin{equation*}
    \begin{split}
        \mathcal{H}^{N-1}\left(\{u=t\}\cap \{|\nabla u|>\sigma\}\right)&\leq\sigma^{1-p}\int_{\{u=t\}}|\nabla u|^{p-1}\,d\mathcal{H}^{N-1}\\
        &= \sigma^{1-p}\int_{\{u>t\}}f(u)\,dx \\
        &\leq |\Omega| \|f\|_{L^{\infty}([0,M])} \sigma^{1-p}.
    \end{split}
\end{equation*}
Thus, 
\begin{equation*}
    \left|u^{-1}(I)\right|\leq M_u(\sigma) + 2\|f\|_{L^{\infty}(([0,M])}|\Omega|\frac{\eps^{1/2}}{\sigma^p} 
,\end{equation*}
as desired.\end{proof}

\begin{proof}[Proof of Theorem~\ref{Theorem 1.1}]
As in ~\cite{CianchiEspositoFuscoTrombetti+2008+153+189}, we define~$G:= (0,t_{\varepsilon,\alpha})\setminus I$ and 
\begin{equation*}
    \Phi(x): = \left(\frac{\mu(u(x))}{\omega_N}\right)^{1/N}.
\end{equation*} It is also convenient to  take into account the function
\begin{equation*}
    \tilde{\mu}(t) := \int_{t}^{M} \chi_G(\tau) \int_{\{u = \tau\}}\frac{d\mathcal{H}^{N-1}}{|\nabla u|}\,d\tau,
\end{equation*}which takes into account only the ``good" level sets of~$u$.

Also, given~$\sigma >0$, we set
\begin{equation}\label{eta}
    \eta := \frac{M_u(\sigma) + 2|\Omega| \|f\|_{L^{\infty}[0,M]} \frac{\eps^{1/2}}{\sigma^{p}}+ \eps^{\alpha N'}}{|\Omega|}
\end{equation}
and 
\begin{equation}\label{t eta}
    t_{\eta}:= \inf\left\{t\geq 0:\, \tilde{\mu}(t)\leq (1-\eta^{1/2})\mu(t)\right\}.
\end{equation}

Then, using the results of Section~\ref{Estimates on level sets} (which are the counterparts of the ones in Section~2 of~\cite{CianchiEspositoFuscoTrombetti+2008+153+189}) and Lemma~\ref{Measure of preimage of I} here, one is able to follow the arguments present in Section~4 of~\cite{CianchiEspositoFuscoTrombetti+2008+153+189} to conclude that, if~$\eta\leq \frac{1}{4}$, there exist positive constants~$\alpha_1(N,p)$, $\alpha_2(N,p)$,
$$
\tilde{C}(N,p,|\Omega|,\mathcal{H}^{N-1}(\partial \Omega),\phi_0,M, \|f\|_{L^{\infty}([0,M])}) , \quad \text{and} \quad
\eps_1(N,p,|\Omega|, \mathcal{H}^{N-1}(\partial \Omega), \phi_0, \|f\|_{L^{\infty}([0,M])}, \mathcal{M}_{0}^{-}) \le 1,
$$
such that, if~$\eps \le \eps_1,$
then 
$$
\min_{x_0\in \mathbb{R}^N}\int_{\mathbb{R}^N} |u(x)-u^*(x_0+x)|\,dx\leq \tilde{C}\left(M_{u}(\sigma)^{\alpha_1}+\eps^{\alpha_2}+ \left(\frac{\eps^{\frac{\alpha_1}{2}}}{\sigma^{p}}\right)^{\alpha_1}+\frac{\eps^{\frac{1}{2}}}{\sigma^p} \right).
$$

As a result, choosing~$\sigma := \eps^{\frac{1}{4p}}$ and using Theorem~\ref{Quantitative estimate of Mu}, we see that
$$
\min _{x_0\in \mathbb{R}^N}\int_{\mathbb{R}^N} |u(x)-u^*(x_0+x)|\,dx\leq C \eps^{\theta}, 
$$
for some positive constants~$C(N,p,|\Omega|,\mathcal{H}^{N-1}(\partial \Omega),\phi_0,M, \|f\|_{W^{1,\infty}([0,M])}, \mathcal{M}_{0}^{-})$ and~$\theta(N,p)$, provided that~$\eps \le \eps_2$ for some positive constant~$\eps_2$ only depending on~$N$ , $p$, $|\Omega|$, $\mathcal{H}^{N-1}(\partial \Omega)$, $\phi_0$ , $M$, $\|f\|_{W^{1,\infty}([0,M])}$, and~$\mathcal{M}_{0}^{-}$. 

This proves~\eqref{Quantitative stability} for~$\eps \le \eps_2$. Moreover, if~$\eps > \eps_2$, then~\eqref{Quantitative stability} trivially holds true.
\end{proof}

\section*{Acknowledgements} 

Serena Dipierro, Giorgio Poggesi, and Enrico Valdinoci are members of the Australian Mathematical Society (AustMS).
Serena Dipierro is supported by the Australian Research Council
Future Fellowship FT230100333
``New perspectives on nonlocal equations''.
Jo{\~a}o Gon\c{c}alves da Silva and
Giorgio Poggesi are supported by the Australian Research Council (ARC) Discovery Early Career Researcher Award (DECRA) DE230100954 ``Partial Differential Equations: geometric aspects and applications''. Jo{\~a}o Gon\c{c}alves da Silva is supported by a Scholarship for International Research Fees at The University of Western Australia.
Enrico Valdinoci is supported by the Australian Laureate Fellowship FL190100081 ``Minimal surfaces, free boundaries and partial differential equations''.

\bibliographystyle{abbrv}

\end{document}